\documentclass{amsart}

\usepackage{fullpage}

 \usepackage{amssymb,latexsym,amsmath,amsthm}

\usepackage[numbers,sort&compress]{natbib}

 \renewcommand{\div}{\mathop{\mathrm{div}}\nolimits}

\newtheorem{open}{Open Problem}

\newtheorem*{thm*}{Theorem A}

\newtheorem*{thm**}{Theorem B}

\newtheorem{thm}{Theorem}[section]

\newtheorem{dfn}{Definition}[section]

\newtheorem{lemma}{Lemma}[section]

\newtheorem{prop}{Proposition}[section]

\newtheorem{cor}{Corollary}[section]

\numberwithin{equation}{section}

\begin{document}

\def\IR{{\mathbb{R}}}

\title{Entire solutions of quasilinear symmetric systems}
\maketitle

\begin{center}
\author{Mostafa Fazly\footnote{The author is partially supported by National Sciences and Engineering Research Council of Canada (NSERC) Discovery Grant\#RES0020463.}
}
\\
{\it\small Department of Mathematical and Statistical Sciences, University of Alberta}\\
{\it\small Edmonton, Alberta, Canada T6G 2G1}\\
{\it\small e-mail: fazly@ualberta.ca}\vspace{1mm}
\end{center}

\begin{abstract}  
We study the following quasilinear elliptic system for all $i=1,\cdots,m$
\begin{equation*} \label{}
   -\div(\Phi'(|\nabla u_i|^2) \nabla u_i )   =   H_i(u)  \quad  \text{in} \ \  \mathbb{R}^n 
  \end{equation*}   
 where $u=(u_i)_{i=1}^m: \mathbb R^n\to \mathbb R^m$  and the nonlinearity $ H_i(u) \in C^1(\mathbb R^m)\to \mathbb R$ is a general nonlinearity. Several celebrated operators such as the  prescribed mean curvature, the Laplacian and  the $p$-Laplacian  operators  fit in the above form,  for appropriate $\Phi$.    We establish  a Hamiltonian identity of the following form for all $x_n\in\mathbb R$
\begin{equation*}\label{}
\int_{\mathbb R^{n-1}} \left( \sum_{i=1}^{m} \left[ \frac{1}{2} \Phi\left(|\nabla u_i|^2\right) - \Phi'\left(|\nabla u_i|^2\right) |\partial_{x_n} u_i|^2 \right] - \tilde H(u)  \right) d x'\equiv C,
\end{equation*}
where $x=(x',x_n)\in\mathbb R^{n}$ and $\tilde H$ is the antiderivative of $H=(H_i)_{i=1}^m$.  This can be seen as a counterpart of celebrated pointwise inequalities provided by Caffarelli, Garofalo and Segala in \cite{cgs} and by Modica in \cite{m}.  

 For the case of system of equations, that is when $m\ge 2$, we show that  as long  as    $$\alpha \ge \alpha^*:=\inf_{s>0}\left\{ \frac{2 s \Phi'(s)}{\Phi(s)}\right\}$$ the function 
$I_\alpha(r):=\frac{1}{r^{n-\alpha}} \int_{B_r}  \sum_{i=1}^{m} \Phi(|\nabla u_i|^2) - 2\tilde H(u)$ is monotone nondecreasing in $r$.  This in particular implies that for the prescribed mean curvature, the Laplacian, the $p$-Laplacian and  operators the function $I_\alpha(r)$ is monotone when  $\alpha\ge\alpha^*=2$,  $\alpha\ge\alpha^*=2$ and $\alpha\ge\alpha^*=p$, respectively.  We call this a weak monotonicity formula since for $m=1$ it is shown in \cite{cgs} that $I_\alpha(r)$ is monotone  when $\alpha\ge 1$,  under certain conditions on $\Phi$.     

 We prove De Giorgi type results for $H$-monotone and stable solutions in  two and three dimensions  when the system is symmetric.  The remarkable point is that gradients of all components of solutions are parallel and the angle  between vectors $\nabla u_i$ and $\nabla u_j$ is precisely $\arccos \left(\frac{|\partial_j H_i|}{\partial_j H_i}\right)$.  In addition, we provide an optimal Liouville theorem regarding radial stable solutions of the above system with a general nonlinearity when the system is symmetric.  We announce several natural open problems in this context as well.   
\end{abstract}

\noindent
{\it \footnotesize 2010 Mathematics Subject Classification}. {\scriptsize 35J45, 35J93, 35J92, 35J50.}\\
{\it \footnotesize Keywords:  Quasilinear elliptic systems,  Hamiltonian identity, monotonicity formula, De Giorgi's conjecture, qualitative properties of solutions, prescribed mean curvature}. {\scriptsize }

\tableofcontents

\section{Introduction}\label{intro}

In \cite{cgs}, Caffarelli, Garofalo and Segala  studied the following class of quasilinear equations arising in geometry 
 \begin{equation} \label{mainc}
   \div(\Phi'(|\nabla u|^2) \nabla u )   =   f(u)  \quad  \text{in} \ \  \mathbb{R}^n , 
  \end{equation} 
  where $f$ is a $C^1(\mathbb R)$ and $\Phi\in C^2(\mathbb R^+)$ satisfies certain conditions that follow.  Note that for the case of $\Phi(s)=s$ the above equation is the standard semilinear elliptic equation.  The prescribed mean curvature equation and  the $p-$Laplacian equation, i.e.
\begin{eqnarray}
\label{curv} && \div\left(\frac {\nabla u}{\sqrt{  1+  |\nabla u|^2}}\right )   =   f(u)  \quad  \text{in} \ \  \mathbb{R}^n ,
\\ \label{plap} && \div\left((\epsilon +|\nabla u|^2)^\frac{p-2}{2}  \nabla u\right)   =   f(u)  \quad  \text{in} \ \  \mathbb{R}^n \ \ \text{for } \epsilon>0,
  \end{eqnarray} 
respectively,  fit  in the form of (\ref{mainc}) where $\Phi$ is given respectively by
\begin{eqnarray}
\label{phim} \Phi(s)&=& 2(\sqrt{1+s}-1), \\
\label{phip}\Phi(s)&=&\frac{2}{p}  \left ( (\epsilon +s)^\frac{p}{2} - \epsilon ^{\frac{p}{2}}  \right).
\end{eqnarray}
  
Throughout this paper  we shall assume that $\Phi(s),\Phi'(s)$ and  $\Phi'(s)+2 \Phi''(s) s$ are positive when $s>0$. In addition, without loss of generality let $\Phi(0)=0$.  Borrowing notations from \cite{cgs}, we shall refer to the following conditions often in this paper.  Note that $\Phi$ in (\ref{phim}) and (\ref{phip})  satisfies these conditions, respectively, 
\begin{enumerate}
\item[]  Condition (A).   There exist positive constants $C_1,C_2$ and $\epsilon\ge 0$   such that $\Phi\in C^2(\mathbb R^+)$ and 
 for every $\eta,\zeta\in\mathbb R^n$ 
\begin{eqnarray}
\label{A1c} &&C_1 (\epsilon+|\eta|)^{-1} \le \Phi'(|\eta|^2) \le C_2 (\epsilon+|\eta|)^{-1},  \\
\label{A2c} &&C_1 (\epsilon+|\eta|)^{-1} |\zeta'|^2 \le \sum_{i,j=1}^m a_{i,j}(\eta) \zeta_i\zeta_j \le C_2 (\epsilon+|\eta|)^{-1} |\zeta'|^2 , 
\end{eqnarray}
where $\zeta'=(\zeta,\zeta_{n+1})\in\mathbb R^{n+1}$ is orthogonal to the vector $(-\eta,1)\in \mathbb R^{n+1}$. \\
\item[] Condition (B).   There exist $p>1$, $\epsilon\ge 0$ and positive constants $C_1,C_2$   such that $\Phi\in C^2(\mathbb R^+)$ that
 for every $\eta,\zeta\in\mathbb R^n$, 
\begin{eqnarray}
\label{A1p} &&C_1 (\epsilon+|\eta|)^{p-2} \le \Phi'(|\eta|^2) \le C_2 (\epsilon+|\eta|)^{p-2} , \\
\label{A2p} &&C_1 (\epsilon+|\eta|)^{p-2} |\zeta|^2 \le \sum_{i,j=1}^m a_{i,j}(\eta) \zeta_i\zeta_j \le C_2 (\epsilon+|\eta|)^{p-2} |\zeta|^2 , 
\end{eqnarray}
where $a_{i,j}$ in (\ref{A2p}) and (\ref{A2c}) are given by 
\begin{equation}\label{aij}
a_{i,j}(\eta): = 2 \Phi''(|\eta|^2) \eta_i\eta_j+ \Phi'(|\eta|^2) \delta_{ij}. 
\end{equation}
\end{enumerate}

  One of the main results provided in \cite{cgs} is the following pointwise inequality.  Note that this is a counterpart of the pointwise estimate given by Modica in \cite{m} for the case of $\Phi(s)=s$. 
    \begin{thm*}\label{caf}
  Suppose that $f\in C^2(\mathbb R)$ with $F\ge 0$ and suppose that one of the following conditions hold
  \begin{enumerate}
\item[(i)] Condition (A) holds and $u\in W^{1,p}_{loc}(\mathbb R^n)\cap L^\infty (\mathbb R^n)$ is a solution to (\ref{mainc}) 
\item[(ii)] Condition (B) holds and $u\in C^2(\mathbb R^n)\cap L^\infty (\mathbb R^n)$ is a solution to (\ref{mainc}) and in additon $|\nabla u| \in L^{\infty}(\mathbb R^n)$. 
    \end{enumerate}
 Then for every $x$ 
\begin{equation}\label{cpoint}
2 \Phi'(|\nabla u|^2) |\nabla u|^2 - \Phi(|\nabla u|^2) \le 2F(u). 
\end{equation}
  \end{thm*}
In particular, the following pointwise estimates hold for specific $\Phi$. 
  \begin{itemize}
  \item Suppose that $\Phi(s)=s$, then 
     \begin{equation}\label{modp}
     |\nabla u|^2 \le 2 F(u) \ \ \text{in} \ \  \mathbb R^n,
     \end{equation}
   where $u$ is a  bounded  solution of the semilinear equation $\Delta u=f(u)$ in $\mathbb R^n$, provided by Modica in \cite{m}. 
    \item Let   $\Phi(s)=2(\sqrt{1+s}-1)$. Then  
       \begin{equation}\label{pointc}
\frac{\sqrt{  1+  |\nabla u|^2}-1}{\sqrt{  1+  |\nabla u|^2}} \le F(u) \ \ \text{in} \ \  \mathbb R^n,
     \end{equation}
 for bounded solutions of the  prescribed mean curvature equation that is $\div\left(\frac {\nabla u}{\sqrt{  1+  |\nabla u|^2}}\right )   =   f(u)$ in $\mathbb R^n$.  
  \item  Suppose that $\Phi(s)=\frac{2}{p}s^{\frac{p}{2}}$. Then  
\begin{equation}\label{pointcp}
|\nabla u|^p \le \frac{p}{p-1} F(u) \ \ \text{in} \ \  \mathbb R^n,
     \end{equation}
  where $u$ is a bounded solution of the $p$-Laplace equation  $\div\left( |\nabla u|^{p-2}  \nabla u \right)   =   f(u)$ in $\mathbb R^n$.  
  \end{itemize}
We study classical solutions of the following quasilinear system of equations 
 \begin{equation} \label{main}
   -\div(\Phi'(|\nabla u_i|^2) \nabla u_i )   =   H_i(u)  \quad  \text{in} \ \  \mathbb{R}^n, 
  \end{equation}   
 where $u=(u_i)_{i=1}^m: \mathbb R^n\to \mathbb R^m$  and $ H_i(u) \in C^1(\mathbb R^m)\to \mathbb R$ for all $i=1,\cdots,m$.   The above system has variational structure  and the  associated   energy functional is given by 
 \begin{equation}\label{Eu=}
E(u)= \int \sum_{i=1}^{m} \frac{1}{2}\Phi(|\nabla u_i|^2) - \tilde H(u),
 \end{equation}
 where $\tilde H$ is defined such that $\partial_i\tilde H(u)=H_i(u)$. Throughout this paper we use the notation $u=(u_i)_{i=1}^m$,  $H(u)=(H_i(u))_{i=1}^m$ and $\partial_j H_i(u)=\frac{\partial H_i(u)}{\partial {u_j}} $. We assume that  $\partial_i H_j (u)\partial_j H_i (u) > 0$ for $1\le i\le j\le m$.    The next definition  is the notion of the symmetric systems, introduced by the author in \cite{mf}.   Symmetric systems play a fundamental role throughout this paper when we deal with the energy functional given in (\ref{Eu=}) and when we study system (\ref{main}) with a general nonlinearity $H(u)$.  Note that for the scalar equation case, that is when $m=1$,  (\ref{main}) is clearly symmetric. 
\begin{dfn}\label{symmetric} We call system (\ref{main}) symmetric if the matrix of partial derivatives of all components of $H$ given by
 \begin{equation} \label{H}
\mathbb{H}:=(\partial_i H_j(u))_{i,j=1}^{m}, 
 \end{equation}
  is symmetric. 
   \end{dfn}
Hamiltonian identities are quite well-known in both  mathematics and physics as important tools to study qualitative behaviour of entire solutions of differential equations and systems. They often directly or indirectly lead to certain properties which could be of great importance in the fields as well,  such as  monotonicity formulae.    Consider the following symmetric system of ordinary differential equations  that is a particular case of (\ref{main}), 
\begin{equation}\label{ODE}
-u_i''= \partial_i H(u) \ \ \text{in} \ \ \mathbb R. 
\end{equation}  
It is straightforward to see that the  following Hamiltonian identity holds for  solutions of (\ref{ODE}) 
\begin{equation}\label{ODE11}
\frac{1}{2} \sum_{i=1}^m u_i'^2 +H(u) \equiv C\ \ \text{in } \ \ \mathbb R,
\end{equation}  
 where  $C$ is a constant.  Equivalently,  one can rewrite (\ref{ODE}) in the form of a first order Hamiltonian system  
\begin{eqnarray*}
 \left\{ \begin{array}{lcl}
\hfill -u_i'&=& \partial_{v_i} \bar H(u,v)   \ \ \text{in}\ \ \mathbb{R},\\   
\hfill -v_i'&=& -\partial_{u_i} \bar H(u,v)   \ \ \text{in}\ \ \mathbb{R},
\end{array}\right.
  \end{eqnarray*}
where $\bar H(u,v) =\frac{1}{2} \sum_{i=1}^m v_i^2+H(u)$. Note that $\bar H(u,v) \equiv C$ on trajectories of solutions.  These equations generalize Newton's third law that is $F = ma$  to system of equations where the momentum is not simply mass times velocity. The Hamiltonian $\bar H(u,v)$ normally  represents the total energy of the system. We refer interested readers to \cite{ar,h} for some original  information regarding physical meaning of the system and to \cite{iv,jm} and references therein for variational theory of Hamiltonian systems.

Gui in \cite{gui}  considered the gradient system $-\Delta u_i= \partial_i H(u)$, that is a higher-dimensional  counterpart of (\ref{ODE}), and  established the following Hamiltonian identity, 
\begin{equation}\label{intHamGui}
 \int_{R^{n-1}} \left[ \frac{1}{2} \sum_{i=1}^m \left( |\nabla_{x'} u_i|^2 - |\partial_{x_n} u_i|^2 \right) - H(u(x)) \right] dx'\equiv C,
 \end{equation}
for $x=(x',x_n)\in\mathbb R^{n-1}\times \mathbb R$.    In this paper, we provide an extension of this inequality for solutions of quasilinear symmetric system (\ref{main}).  One might expect, at the first glance, that just  replacing derivative terms with $ \Phi(|\nabla_{x'} u_i|^2)- \Phi(  |\partial_{x_n} u_i|^2 )$ could simply give the Hamiltonian identity for solutions of (\ref{main}).  Instead,  the identity follows the structure of the pointwise estimate provided by Caffarelli, Garofalo and Segala in \cite{cgs} and it is of the form
\begin{equation}\label{intHam}
\int_{\mathbb R^{n-1}} \left( \sum_{i=1}^{m} \left[ \frac{1}{2} \Phi\left(|\nabla u_i|^2\right) - \Phi'\left(|\nabla u_i|^2\right) |\partial_{x_n} u_i|^2 \right] - \tilde H(u(x))  \right) d x'\equiv C.
\end{equation}
Note that when $\Phi(s)=s$, the identity (\ref{intHam})  recovers (\ref{intHamGui}). This  then  explains why the difference of partial derivatives that is $|\nabla_{x'} u_i|^2 - |\partial_{x_n} u_i|^2$ appears in (\ref{intHamGui}).    

If we set $\Phi$ to be the ones given in (\ref{phim}) and (\ref{phip}) we can have the Hamiltonian identity for the prescribed mean curvature equation and  the $p$-Laplacian equation, respectively.  Let us mention  this remarkable point again  that the Hamiltonian identity (\ref{intHam}) has a very similar structure  as  pointwise estimates (\ref{cpoint}) and  (\ref{modp}), provided by Caffarelli, Garofalo and Segala in \cite{cgs} and by Modica in \cite{m}.     Therefore,  (\ref{intHam}) can be seen as a counterpart of (\ref{cpoint}) for the case of system of equations, i.e. $m\ge 2$.  

The Hamiltonian identity (\ref{intHam}) motivates us to look for a monotonicity formula for solutions of (\ref{main}). So, set 
\begin{equation}\label{intI}
I_\alpha(r):=\frac{1}{r^{n-\alpha}} \int_{B_r} \left[ \sum_{i=1}^{m} \Phi(|\nabla u_i|^2) - 2\tilde H(u)\right]. 
\end{equation}
For the case of scalar equations, that is when $m=1$, it is proved by Caffarelli, Garofalo and Segala in \cite{cgs} that when $\Phi$ satisfies one of conditions (A) or (B) then the function $I_\alpha(r)$ is monotone nondecreasing in $r$ when $\alpha\ge 1$.  They have used the pointwise inequality (\ref{cpoint}) to establish this monotonicity formula.   For the case of $m \ge 2$, we show that $I_\alpha(r)$ is monotone nondecreasing in $r$ when 
\begin{equation}\label{intalphaphi}
\alpha \ge \alpha^*:=\inf_{s>0}\left\{ \frac{2 s \Phi'(s)}{\Phi(s)}\right\}.
 \end{equation}
We call this a weak monotonicity formula since the constant $\alpha^*$ must be greater than one, due to some general assumptions on $\Phi$.   To clarify this,  define an auxiliary function $h(s):=-2\Phi'(s) s +\alpha\Phi( s)$ in the light of (\ref{intalphaphi}).  Note that from assumptions on $\Phi$, i.e. $2s\Phi''(s)+\Phi'(s)>0$ when $s>0$ and $h(0)=\alpha \Phi(0)=0$ one can see that  $h'(s)=-[2s\Phi''(s)+\Phi'(s)]+(\alpha-1) \Phi'(s)$ is negative  when $\alpha \le  1$.  For certain functions $\Phi$,  one can get the Laplacian, the $p$-Laplacian and the prescribed mean curvature operators and then the function $I_\alpha(r)$ is monotone in $r$ when  $\alpha\ge\alpha^*=2$,  $\alpha\ge\alpha^*=p$ and $\alpha\ge\alpha^*=2$, respectively,  see Corollary \ref{corphis}.    On the other hand, for both cases of scalar equations and system of equations, i.e. $m\ge 1$, it is shown in  Theorem \ref{enopthm} that the following upper bound on the energy holds,
\begin{equation}\label{intenergybound}
\int_{B_{R}}  \left[  \sum_{i=1}^m \Phi(|\nabla u_i|^2) -2 \tilde H(u)+2\tilde H(a) \right] d x \le C R^{n-1},
\end{equation}
where  $\lim_{x_n\to \infty} u_i( x',x_n)=a_i$ for all $ x=( x', x_n)\in\mathbb{R}^{n}$ and $a=(a_i)_{i=1}^m$.  This implies that for the case of system of equations, $m\ge2$, the strong monotonicity formula, that is when $\alpha \ge 1$,  should hold just like in the case of  scalar equations for  $m=1$. This remains as an open problem. Note also that conditions (A) and (B) are not necessary for our monotonicity formula when $m\ge 2$.

We apply monotonicity formulae to establish Liouville theorems for solutions of (\ref{main}) with a finite energy.  We refer interested readers to Alikakos in \cite{ali,ali2} and to Alikakos and Fusco in \cite{af}, to  Caffarelli, Garofalo and Segala in \cite{cgs} and  to Farina in \cite{f1,f3} regarding Liouville theorems for various equations and systems with a finite energy.   Note that the above monotonicity formulae, in both weak and strong forms,  are related to the ones given for harmonic maps
 by Schoen and Uhlenbeck in \cite{su} and  for
minimal surfaces by Simon in \cite{si}, by Ecker in \cite{ec} and  by Schoen in \cite{sc} and for elliptic  equations   by Caffarelli and Lin in \cite{cl},  by Modica in \cite{m3} and references therein.

 Regarding the scalar equation case, in this context,  monotonicity of a solution $u$ is straightforward to define and it refers to solutions that are monotone in one direction, e.g.  when $\partial_{x_n} u_i$ does not change sign, see \cite{ac,aac,gg1,gg2,dkw,sav,cgs,fsv,f1,f2,f3, mm, DeGiorgi,big} and references therein.   However, the notion of monotonicity of solutions for the case of system of equations, that is when $m\ge 2$, seems to be slightly more sophisticated.   Ghoussoub and the author in \cite{fg} introduced the following concept of  monotonicity for the case of system of equations. Note that the sign of  partial derivatives of the  nonlinearity $H$ could potentially have an impact on the monotonicity of solutions.  This motivates us to call this notion as $H$-monotonicity.

 %Consider the linearized operator 
%$L_u=(L_{u_i})_{i=1}^m$ that is defined via duality 
%\begin{eqnarray}\label{LL}
%L_{u_i}(\zeta)[\eta] &=& \int_{\mathbb R^n} \Phi'(|\nabla u_i|^2)  (\nabla \zeta_i,\nabla \eta_i) \\&& \nonumber+ 2 \int_{\mathbb R^n}\Phi''(|\nabla u_i|^2)  (\nabla u_i,\nabla \zeta_i)(\nabla u_i,\nabla \eta_i)
%\end{eqnarray}

 \begin{dfn}\label{Hmon} A solution $u=(u_k)_{k=1}^m$ of (\ref{main}) is said to be  $H$-monotone if the following holds,
\begin{enumerate}
 \item[(i)] For every $1\le i \le m$, each $u_i$ is strictly monotone in the $x_n$-variable (i.e., $\partial_{x_n} u_i\neq 0$).

\item[(ii)]  For all $i\le j$,  we have 
  \begin{equation}\label{huiuj}
\hbox{$\partial_j H_i(u) \partial_{x_n} u_i(x) \partial_{x_n} u_j (x) > 0$  for all $x\in\mathbb {R}^n$.}
\end{equation}
\end{enumerate}
See \cite{fg,mf} for more details.  
\end{dfn}
Note also that in the assumption (ii) of the $H$-monotonicity each of $\partial_j H_i(u)$, $ \partial_{x_n} u_i(x) $ and $\partial_{x_n} u_j (x)$ has a fixed sign and the multiplication must be positive. This  implies a combinatorial assumption on the system (\ref{main}).  We refer to systems that admit such  an assumption  as {\it orientable} systems. For an example,  consider $m=2$ then for cross type solutions, i.e.   $\partial_{x_n} u_1>0$ and $\partial_{x_n} u_2<0$,  we are required  to set $\partial_1 H_2(u),\partial_2 H_1(u)<0$.  If we set $H_1(u)=H_2(u)=-\frac{1}{2}u^2_1 u^2_2$ then this gives a two component system of equations that arrises in Bose-Einstein condensates, see \cite{blwz} and references therein. 

The next definition  is the notion of stable solutions for the case of system of equations. 
\begin{dfn} \label{stable}
A solution $u=(u_k)_{k=1}^m$ of (\ref{main}) is called stable when there exists a sequence of  functions   $\phi=(\phi_k)_{k=1}^m$ such that each $\phi_i$ does not change sign and  $\partial_j H_i(u)  \phi_j \phi_i>0$ for all $i,j=1,\cdots,m$. In addition,  $\phi$ satisfies   the following 
 \begin{equation} \label{L}
-\div(\mathcal A(\nabla u_i) \nabla \phi_i)= \sum_{j=1}^m \partial_j H_i(u)  \phi_j    \ \ \ \text{in}\ \ \mathbb R^n,
  \end{equation}
where  for any $\eta \in\mathbb R^n$ the matrix $\mathcal A(\eta)$   is defined by $\mathcal A(\eta):=(a_{i,j}(\eta))_{i,j=1}^n$ for $a_{i,j}(\eta)$ in (\ref{aij}). 
\end{dfn} 
Let us mention that the notion of stability can be given for weak solutions as 
\begin{equation}\label{weakstab}
 \int \mathcal A(\nabla u_i) \nabla \phi_i \cdot \nabla \zeta_i= \sum_{j=1}^m \partial_j H_i(u)  \phi_j \zeta_i, 
 \end{equation}  where $\zeta=(\zeta_i)_{i=1}^m$ is a sequence of test functions. Accordingly one can see that   all results provided in the present paper are valid for weak solution as well.   For the sake of simplicity in notation, we present our results for classical solutions.  We refer to \cite{CR} and references therein for the use of stability for nonlinear elliptic eigenvalue
problems. 

 Since (\ref{L}) is a linearization of (\ref{main}), one can see that every $H$-monotone solution is a stable solution via differentiating (\ref{main}) with respect to $x_n$.   The notion of stability as well as the  monotonicity formula for $I_\alpha(r)$ when $\alpha\ge 1$  and the pointwise inequality  (\ref{modp}), provided by Modica,  play key role in settling the De Giorgi's conjecture (1978),  see \cite{DeGiorgi}.  The conjecture states that bounded monotone solutions of Allen-Cahn equation are one-dimensional solutions at least up to eight dimensions. There is an affirmative answer to this conjecture for almost all dimensions.  More precisely,  for two dimensions  Ghoussoub and Gui in \cite{gg1} and for three dimensions Ambrosio and Cabr\'{e} in \cite{ac} and with Alberti in \cite{aac} gave a proof to this conjecture not only for Allen-Cahn equation but also for any equation of the form $-\Delta u=f(u)$ where $f$ is a general nonlinearity that is locally Lipschitz.   For dimensions $4\le n \le 8$ there are various partial results under certain extra (natural) assumptions on solutions by Ghoussoub and Gui in \cite{gg2},  by Savin in \cite{sav} and references therein.  Note that there is an example by del Pino, Kowalczyk and Wei in \cite{dkw}  showing that eight dimensions is the critical dimension.   In two dimensions  regarding  the De Giorgi's conjecture,  we refer to  Farina,  Sciunzi and Valdinoci in \cite{fsv} for a geometrical approach and  to Modica and Mortola in \cite{mm} for some partial results  under the additional assumption that the level sets of solutions are the graphs of an equi-Lipschitzian family of functions.   

  In \cite{fsv,cgs,dg}, authors considered quasilinear scalar equations of the form of (\ref{main}) when $m=1$ and provided one-dimensional symmetry and De Giorgi type results.      Note that Ghoussoub and the author in \cite{fg} provided De Giorgi type results for elliptic systems of the form $-\Delta u_i=\partial_i H(u)$ in lower dimensions for a general nonlinearity $H$.

In this paper, we first provide a geometric Poincar\'{e} inequality and a linear Liouville theorem for stable and $H$-monotone solutions of the quasilinear system (\ref{main}). Then we apply these to conclude De Giorgi type results for $H$-monotone and stable solutions in two and three dimensions  when the system is symmetric.  For coupled systems,  that is when not all $\partial_j H_i$ vanish for $1\le i< j\le m$, it is natural to expect that there should be a relation between two arbitrary components $u_i$ and $u_j$.    In this regard, we show that gradients of all components of solutions are parallel and the angle in between $\nabla u_i$ and $\nabla u_j$ is precisely $\arccos \left(\frac{|\partial_j H_i|}{\partial_j H_i}\right)$ when $\partial_j H_i \neq 0$. This is a  consequence of the geometric Poincar\'{e} inequality, see Theorem \ref{lempoin}. 
 
 The main focus of the present paper is the study of qualitative properties of solutions of system (\ref{main}) with a general nonlinearity.   In this paper,  we prove a Liouville theorem for bounded stable  solutions of (\ref{main}) in dimensions $n \le 4$ for a general nonlinearity $H=(H_i)_{i=1}^m$  whenever each $H_i$ is nonnegative.  To do so,  we suppose that $\Phi$ satisfies either condition (A) or (B).  Note that for the case of semilinear equations similar results are given by Dupaigne and Farina in \cite{df} and for the case of semilinear systems by Ghoussoub and the author in \cite{fg}. 
 In addition,  we give a classification of radial stable solutions of symmetric system (\ref{main}) when $\Phi(s)=\frac{2}{p} s^\frac{p}{2}$ for all $m\ge 1$.  More precisely, we show that there exists a positive constant $C_{n,m,p}$  such that for any $r$, the following pointwise lower bound holds, 
  \begin{equation}\label{intlower}
 \sum_{i=1}^{m} | u_i(r)| \ge C_{n,m,p}   \left\{
                      \begin{array}{ll}
                        r^{\frac{1}{p} \left(p+2-n+2\sqrt{\frac{n-1}{p-1}}\right)}, & \hbox{if $n \neq \frac{4p}{p-1}+p$,} \\
                      \log r, &  \hbox{if $n =  \frac{4p}{p-1}+p$.} 
                           \end{array}
                    \right.
                    \end{equation}
  This in particular implies that bounded radial stable solutions must be constant in dimensions $1 \le n< \frac{4p}{p-1}+p$.    The notion of symmetric systems seems to be essential to study (\ref{main}) with a general nonlinearity.      Note also that the critical dimension    $ n= \frac{4p}{p-1}+p$ for radial solutions  is much higher than the dimension $n=4$ derived for not necessarily radial solutions.  Let us mention that  for the case of semilinear equations, that is $\Phi(s)=s$ and $m=1$, it is proved by Cabr\'{e}-Capella \cite{cc1,cc2} and Villegas \cite{sv} that any bounded radial stable solution of (\ref{main})  has to be constant provided $1\le n < 10$ when $H\in C^1(\mathbb R)$ is a general nonlinearity.  In addition, for the case of scalar equation and when $\Phi(s)=\frac{2}{p} s^\frac{p}{2}$,  a counterpart of the above Liouvillle theorem is provided in \cite{CS,ces}.

Here is how this paper is  organized. Shortly after,  in Section \ref{secham} we provide a Hamiltonian identity for solutions of system (\ref{main}).  We also prove  monotonicity formulae and we apply it to establish a Liouville theorem for solutions with finite energy.  A few open problems are provided in this section as well.   Section \ref{secpoin} is devoted to some  estimates needed to prove De Giorgi type results and Liouville theorems in next sections.  We start the section with a stability inequality and then we apply this inequality to establish a geometric Poincar\'{e} inequality.   In Section \ref{secde}, we establish De Giorgi type results for $H$-monotone and stable solutions of symmetric system (\ref{main}). In addition, we apply the  geometric Poincar\'{e} inequality, provided in Section \ref{secde}, to find a relation between gradients of all components of solutions of (\ref{main}).   Finally in Section \ref{secop},  we  prove  Liouville theorems for stable solutions of (\ref{main}) with a general nonlinearity,  with on case requiring the solutions to be also radial.    The concept of symmetric systems seems to be crucial to prove such an optimal  Liouville theorem for radial solutions.

\section{Hamiltonian identities  and monotonicity formulae }\label{secham}
We start this section by the following Hamiltonian identity.   
\begin{thm}\label{hamthm}
Suppose that $u=(u_i)_{i=1}^m$ is a solution of (\ref{main}) and let $x=(x',x_n)\in\mathbb R^{n-1}\times \mathbb R$. Then there exists a constant $C$ such that  the following Hamiltonian identity holds for every $x_n \in \mathbb R$
\begin{equation}\label{hamil}
\int_{\mathbb R^{n-1}} \left( \sum_{i=1}^{m} \left[ \frac{1}{2} \Phi(|\nabla u_i|^2) - \Phi'(|\nabla u_i|^2) |\partial_{x_n} u_i|^2 \right] - \tilde H(u)  \right) d x'\equiv C,
\end{equation}
when the above integral is finite for at least one value of $x_n$ and in addition the integral in (\ref{integ})
below tends to zero as $R$ goes to infinity along a sequence.
\end{thm}

\begin{proof} Suppose that $x=(x',x_n)\in\mathbb R^n$ and assume that $B_R(0)$ is a ball of radius $R$ in $\mathbb R^{n-1}$.  Define $\Gamma:\mathbb R\to \mathbb R$ as  
\begin{equation}\label{ham}
\Gamma_R(x_n):=\int_{B_R(0)} \left( \sum_{i=1}^{m} \left[ \frac{1}{2} \Phi(|\nabla u_i|^2) - \Phi'(|\nabla u_i|^2) |\partial_{x_n} u_i|^2 \right] - \tilde H(u)  \right)  dx'.
\end{equation}
Differentiating $\Gamma$ with respect to $x_n$ we get 
\begin{eqnarray}\label{}
\nonumber \Gamma'_R(x_n) &=&   \sum_{i=1}^{m} \int_{B_R(0)}  \{  \frac{1}{2} \partial_{x_n} \left[  \Phi(|\nabla u_i|^2) \right]  -  \partial_{x_n} \left[ \Phi'(|\nabla u_i|^2) |\partial_{x_n} u_i|^2 \right] 
\\&&\nonumber - H_i(u) \partial_{x_n} u_i
   \} dx'
\\&=:& \label{1gamma} 
 \sum_{i=1}^{m} \int_{B_R(0)}   \left\{ 
 \Gamma^1(x) + \Gamma^2 (x)+\Gamma^3(x) \right\} dx'. 
  \end{eqnarray}
In what follows we simplify the above three terms,  appeared in the right-hand side of (\ref{1gamma}).  Note that 
\begin{equation}\label{gammae}
\partial_{x_n} \left[  \Phi(|\nabla u_i|^2) \right]=2 \Phi'(|\nabla u_i|^2) \nabla u_i \cdot \nabla \partial_{x_n}u_{i}.
\end{equation}
Therefore, 
\begin{equation}\label{gamma1}
 \Gamma^1(x) = \Phi'(|\nabla u_i|^2) \nabla u_i \cdot \nabla \partial_{x_n}u_{i}.
\end{equation}
Similarly, 
\begin{eqnarray}\label{gamma2}
 \Gamma^2(x) &=&-  2 \Phi''(|\nabla u_i|^2) \nabla u_i \cdot \nabla \partial_{x_n}u_{i} |\partial_{x_n}u_{i}|^2 \\&&\nonumber - 2 \Phi'(|\nabla u_i|^2) \partial_{x_n}u_{i} \partial^2_{x_nx_n}u_{i}.
\end{eqnarray}
 We now apply (\ref{main}) to  simplify $\Gamma^3$, 
\begin{eqnarray}\label{gamma3}
\nonumber \Gamma^3(x) = \div(\Phi'(|\nabla u_i|^2) \nabla u_i )  \partial_{x_n} u_i &=&   \div_{x'} \left(\Phi'(|\nabla u_i|^2) \nabla_{x'} u_i \right)  \partial_{x_n} u_i  
 \\&&\nonumber  + 2\Phi'' \left(|\nabla u_i|^2\right)   \nabla u_i \cdot \nabla \partial_{x_n}u_{i} |\partial_{x_n}u_{i}|^2 
\\&& + \Phi'\left(|\nabla u_i|^2\right) \partial_{x_n}u_{i} \partial^2_{x_n x_n}u_{i} . 
\end{eqnarray}
Adding (\ref{gamma1}),  (\ref{gamma2}) and (\ref{gamma3}) we get 
\begin{eqnarray}\label{gamma123}
 \Gamma^1(x) + \Gamma^2 (x)+\Gamma^3(x)& =&  \div_{x'} \left(\Phi'(|\nabla u_i|^2) \nabla_{x'} u_i \right)  \partial_{x_n} u_i \\&&\nonumber  + \Phi'\left(|\nabla u_i|^2\right) \partial_{x_n}u_{i} \partial^2_{x_n x_n}u_{i} . 
 \end{eqnarray}
Substituting (\ref{gamma123}) in (\ref{1gamma}) and applying the divergence theorem we obtain
\begin{eqnarray}\label{gammapp}
\Gamma'_R(x_n) &=& \sum_{i=1}^{m} \int_{\partial B_R(0)} \Phi'\left(|\nabla u_i|^2\right) \partial_{\nu_{x'}}u_{i} \partial_{x_n}u_{i}. 
 \end{eqnarray}
Suppose that the integral in (\ref{hamil}) is finite when $x_n=0$. Then, 
\begin{equation}\label{integ}
\Gamma_R(x_n) -\Gamma_R(0) = \sum_{i=1}^{m}  \int_0^{x_n} \int_{\partial B_R(0)} \Phi'\left(|\nabla u_i|^2\right) \partial_{\nu_{x'}}u_{i} \partial_{x_n}u_{i} . 
 \end{equation}
Taking the limit of the above when $R\to \infty$ finishes the proof. 

\end{proof}

When $\Phi$ is the identity function, the system of equations (\ref{main}) is a semilinear system of the following form  $$ - \Delta u_i  =   H_i(u)  \quad  \text{in} \ \  \mathbb{R}^n.$$
Therefore, Theorem \ref{hamthm} implies that the following  Hamiltonian identity  holds,
\begin{equation}\label{hamili}
\int_{\mathbb R^{n-1}} \left( \sum_{i=1}^{m} \frac{1}{2}\left[  |\nabla_{x'} u_i|^2) - |\partial_{x_n} u_i|^2 \right] - \tilde H(u)  \right) d x'\equiv C.
\end{equation}
Note that (\ref{hamili}) is given by Gui in \cite{gui}.  Here we have Hamiltonian identities for the mean curvature system as well as  the $p$-Laplacian  system. 

\begin{cor} Suppose that assumptions of Theorem \ref{hamthm} hold. 
\begin{enumerate}
\item[(i)]  Let $\Phi(s)=2(\sqrt{s+1}-1)$.  Then  (\ref{main}) reads
$$-   \div \left(\frac{ \nabla u_i}{ \sqrt{1+|\nabla u_i|^2}  }\right )   =   H_i(u)  \quad  \text{in} \ \  \mathbb{R}^n.$$
For any $x_n\in\mathbb R$, this  Hamiltonian identity holds, 
\begin{equation}\label{hamilc}
\int_{\mathbb R^{n-1}} \left( \sum_{i=1}^{m} \left[ \frac{1+|\nabla _{x'} u_i|^2 -\sqrt{1+|\nabla u_i|^2} }{\sqrt{1+|\nabla u_i|^2}}   \right] - \tilde H(u)  \right) d x'\equiv C.
\end{equation}
\item[(ii)] Let $\Phi(s)=\frac{2}{p} s^{\frac{p}{2}}$.    Then  (\ref{main}) reads 
$$-   \div \left( |\nabla u_i|^{p-2} \nabla u_i \right )   =   H_i(u)  \quad  \text{in} \ \  \mathbb{R}^n.$$
For any $x_n\in\mathbb R$, this  Hamiltonian identity holds, 
\begin{equation}\label{hamilp}
\int_{\mathbb R^{n-1}} \left( \sum_{i=1}^{m} |\nabla  u_i|^{p-2}\left[ \frac{1}{p} |\nabla _{x'} u_i|^2 -\frac{p-1}{p} |\partial_{x_n} u_i|^2 \right] - \tilde H(u)  \right) d x'\equiv C.
\end{equation}
\end{enumerate}
where $C$ is a constant. 
\end{cor}
Note that for a specific nonlinearity of the form $H_i(u)=u_i\left(1-\sum_{i=1}^m u^2_i\right)$, the original  system of equations (\ref{main}) is    
$$    -\div(\Phi'(|\nabla u_i|^2) \nabla u_i )   =  u_i\left(1-\sum_{i=1}^m u^2_i\right)  \quad  \text{in} \ \  \mathbb{R}^n . 
$$
This is a quasilinear Ginzburg-Landau system  where $u=(u_i)_{i=1}^m$ for $u_i:\mathbb R^n\to \mathbb R$.  For the semilinear case and $m=2$,  see \cite{bmr,f3}.  Since  $\tilde H=-\frac{1}{4} (1-\sum_{i=1}^m u^2_i)^2$ is an antiderivative of $H$, the following Hamiltonian identity holds, as long as conditions of Theorem \ref{hamthm} are satisfied, for any $x_n\in\mathbb R$ 
 \begin{equation}\label{hamilgl}
\int_{\mathbb R^{n-1}} \left\{ \sum_{i=1}^{m} \left[ \frac{1}{2} \Phi(|\nabla u_i|^2) - \Phi'(|\nabla u_i|^2) |\partial_{x_n} u_i|^2 \right] +\frac{1}{4} \left(1-\sum_{i=1}^m u^2_i\right)^2  \right\} d x'\equiv C. 
\end{equation}
For the rest of this section we study monotonicity formulae for solutions of system (\ref{main}).  Consider the following function $I_\alpha(r)$ for any $r>0$ 
\begin{equation}\label{Ir} 
I_\alpha(r):=\frac{1}{r^{n-\alpha}} \int_{B_r} \left[ \sum_{i=1}^{m} \Phi(|\nabla u_i|^2) - 2\tilde H(u)\right]. 
\end{equation}
For the case of scalar equation, that is when $m=1$, the following monotonicity formula holds. Note that this is a direct consequence of Theorem A. 
\begin{thm**}\cite{cgs}
Suppose that $m=1$ and $u$ is a solution of (\ref{mainc}).  In addition,  suppose that assumptions of Theorem A  hold. Then the functional $I_\alpha(r)$  when $\alpha\ge 1$ is a monotone nondecreasing  function of  $r$. 
\end{thm**}

For the case of system of equations, that is when $m\ge2$, we  provide a weaker version of the monotonicity formula provided in Theorem B,  under certain lower bounds on $\alpha$ depending on $\Phi$.        

\begin{thm}\label{monoh}
Let $u=(u_i)_{i=1}^m$ be a solution of (\ref{main}) when $m\ge2$ and $\tilde H(u) \le 0$.   Suppose that there exists a constant $\alpha$ such that 
\begin{equation}\label{alphaphi}
\alpha \ge \alpha^*:=\inf_{s>0}\left\{ \frac{2 s \Phi'(s)}{\Phi(s)}\right\}.
 \end{equation}
Then,  the functional $I_\alpha(r)$ is a monotone nondecreasing function of $r$. In particular,
\begin{equation}\label{mmI}
I_\alpha'(r)  \ge \frac{2}{r^{n-\alpha}} \int_{\partial B_r}  \sum_{i=1}^{m} \Phi'(|\nabla u_i|^2) (\partial_r u_i)^2 
-\frac{2\alpha}{r^{n-\alpha+1}} \int_{B_r}    \tilde H(u).
\end{equation}
\end{thm}
Unlike Theorem B, conditions (A) and (B) do not appear in assumptions of the above theorem for the case of system of equations.  This implies that Theorem \ref{monoh} is valid for a larger class of nonlinearities $\Phi$ compared to Theorem B.  However,  as mentioned in Section \ref{intro}, the constant $\alpha^*$ must be greater than one, due to assumptions on $\Phi$.   For the sake of convenience of readers we clarify this here as well. Consider the auxiliary function $h(s):=-2\Phi'(s) s +\alpha\Phi( s)$ regarding terms appeared in  (\ref{alphaphi}).  Note that from assumptions on $\Phi$, i.e. $2s\Phi''(s)+\Phi'(s)>0$ when $s>0$ and $h(0)=\alpha \Phi(0)=0$ one can see that  $h'(s)=-[2s\Phi''(s)+\Phi'(s)]+(\alpha-1) \Phi'(s)$ is negative  when $\alpha \le  1$.    We now compute $\alpha^*$, provided in (\ref{alphaphi}),  for various choices of $\Phi$. 

\begin{cor}\label{corphis} 
Suppose that $m\ge 2$ and $\tilde H \le 0$. Consider the following particular functions $\Phi$. 
\begin{enumerate}
\item[(i)] Let $\Phi(s)=s$ and $u=(u_i)_{i=1}^m$ be a solution of the semilinear system of equations
\begin{equation} \label{mainpct}
  - \Delta u_i  =   H_i(u)  \quad  \text{in} \ \  \mathbb{R}^n . 
  \end{equation} 
  Then for all $\alpha \ge \alpha^*=2$, the function $I_\alpha(r)$ is monotone nondecreasing in $r$.   
  
\item[(ii)] Let $\Phi(s)=2(\sqrt{1+s} -1)$ and $u=(u_i)_{i=1}^m$ be a solution  of the mean curvature system of equations 
\begin{equation} \label{mainpc}
-   \div \left(\frac{ \nabla u_i}{ \sqrt{1+|\nabla u_i|^2}  }\right )   =   H_i(u)  \quad  \text{in} \ \  \mathbb{R}^n .
  \end{equation} 
    Then for all $\alpha \ge \alpha^*= 2$, the function $I_\alpha(r)$ is monotone nondecreasing in $r$.  
\item[(iii)] Let $\Phi(s)= \frac{2}{p} s^{\frac{p}{2}}$ and $u=(u_i)_{i=1}^m$ be a solution  of $p$-Laplacian system of equations 
\begin{equation} \label{mainpc}
 -  \div \left( |\nabla u_i|^{p-2} \nabla u_i  \right )   =   H_i(u)  \quad  \text{in} \ \  \mathbb{R}^n .
  \end{equation} 
    Then for all $\alpha \ge \alpha^*=p$, the function $I_\alpha(r)$ is monotone nondecreasing in $r$.  
\end{enumerate}
\end{cor}
To provide a proof for Theorem \ref{monoh} we present a few technical estimates. We follow a classical argument regarding Pohozaev and Rellich type identities \cite{poh} to prove the following identity on a ball of radius $r$.   
\begin{lemma}\label{poho} 
Suppose that $u=(u_i)$ is a solution of (\ref{main}) then 
  \begin{eqnarray*}\label{}
-n \int_{B_r} \sum_{i=1}^{m} \Phi( |\nabla u_i|^2 ) &=& 2 r \int_{\partial B_r} \sum_{i=1}^{m} \Phi'(|\nabla u_i|^2) (\partial_r u_i)^2 -r \int_{\partial B_r} \sum_{i=1}^{m} \Phi(|\nabla u_i|^2) \\&& - 2 \int_{B_r} \sum_{i=1}^{m} \Phi'(|\nabla u_i|^2) |\nabla u_i|^2 - 2n \int_{B_r} H(u) +2 r \int_{B_r} H(u) . 
  \end{eqnarray*}
\end{lemma}
\begin{proof} Multiply the $i^{th}$ equation of (\ref{main}) with $x\cdot \nabla u_i$ and then apply the divergence theorem to get 
  \begin{eqnarray*}\label{}
\int_{\partial B_r} \Phi(|\nabla u_i|^2) &=& 2 \int_{\partial B_r} \Phi'(|\nabla u_i|^2) (\partial_r u_i)^2 - \frac{2}{r} \int_{B_r} \Phi'(|\nabla u_i|^2)  |\nabla u_i|^2 \\&&
 - \frac{2}{r}\int_{B_r}   (x \cdot \nabla u_i) \div\left( \Phi'(|\nabla u_i|^2) \nabla u_i \right) +\frac{n}{r} \int_{B_r} \Phi(|\nabla u_i|^2) ,
  \end{eqnarray*}
for each $i=1,\cdots,m$.  Doing some straightforward computations as well as applying (\ref{main})  gives the desired result. 

\end{proof}

We are now ready to provide a proof for Theorem \ref{monoh}. 
\\
\\
\noindent {\it Proof of Theorem \ref{monoh}}.  Differentiating $I(r)$, given by (\ref{Ir}), with respect to $r$ gives 
 \begin{eqnarray*}\label{}
I_\alpha'(r) r^{n-\alpha+1} &=&  (\alpha-n)   \int_{B_r} \sum_{i=1}^{m} \Phi(|\nabla u_i|^2) - 2\tilde H(u) 
\\&&+ r  \int_{\partial B_r} \sum_{i=1}^{m} \Phi(|\nabla u_i|^2) - 2\tilde H(u).
  \end{eqnarray*}
 Substituting the value of $-n \int_{B_r} \sum_{i=1}^{m} \Phi( |\nabla u_i|^2 ) $,  as it is provided in Lemma \ref{poho},  one can show that  
 \begin{eqnarray*}\label{}
I_\alpha'(r) r^{n-\alpha+1} &=& 2r \int_{\partial B_r}  \sum_{i=1}^{m} \Phi'(|\nabla u_i|^2) (\partial_r u_i)^2 
\\&&+ \int_{B_r} \left[ \sum_{i=1}^{m} \left( -2\Phi'(|\nabla u_i|^2)  |\nabla u_i|^2 + \alpha \Phi(|\nabla u_i|^2) \right) -2 \alpha  \tilde H(u) \right].
  \end{eqnarray*}
The rest of the proof is straightforward. 

\hfill $ \Box$

As an immediate consequence of Theorem \ref{monoh} we have the  following Liouville theorem for solutions of (\ref{main}) with a finite energy.  
\begin{thm}\label{LioufiniteE}
Suppose that assumptions of Theorem \ref{monoh} hold.  Assume also that $u=(u_i)_{i=1}^m$ has a finite energy that is  
\begin{equation}\label{Ilimit}
 \int_{\mathbb R^n} \left[ \sum_{i=1}^{m} \Phi(|\nabla u_i|^2) - 2\tilde H(u) \right] dx <\infty.
 \end{equation}
Then  each $u_i$ must be constant in dimensions $n \ge \alpha$ for $i=1,\cdots,m$.  
 \end{thm}
\begin{proof} First suppose that $n> \alpha$.  From  Theorem  \ref{monoh}  for any $R>r$ we have 
\begin{eqnarray*}\label{}
0\le I_\alpha(r) &=&\frac{1}{r^{n-\alpha}} \int_{B_r}  \left[ \sum_{i=1}^{m} \Phi(|\nabla u_i|^2) - 2\tilde H(u)\right]
\\&\le& \frac{1}{R^{n-\alpha}} \int_{B_R} \left[\sum_{i=1}^{m} \Phi(|\nabla u_i|^2) - 2\tilde H(u)\right]
\\&\le & \frac{1}{R^{n-\alpha}} \int_{\mathbb R^n}  \left[\sum_{i=1}^{m} \Phi(|\nabla u_i|^2) - 2\tilde H(u) \right] . 
  \end{eqnarray*}
Sending $R\to \infty$, we get the desired result.   Now suppose that $n= \alpha$. Again from Theorem  
\ref{monoh} we have 
\begin{equation}
r I'_\alpha(r)  \ge  -2\alpha \int_{B_r} \tilde H(u) .
\end{equation}
From this for any $r >\bar r$, where $\bar r$ is fixed, we have 
\begin{equation}\label{Ira}
 I_\alpha (r)  \ge I_\alpha(\bar r)  -2\alpha \ln \left(\frac{r}{\bar r}\right) \int_{B_{\bar r}} \tilde H(u). 
\end{equation}
Note that (\ref{Ilimit}) implies that $\lim_{r\to\infty} I_\alpha(r)<\infty$. From this and  (\ref{Ira}) we conclude that $\tilde H=0$. The fact that each component $u_i$ is harmonic together with (\ref{Ilimit}) completes the proof. 

\end{proof}

Another consequence of the monotonicity formula, given in (\ref{mmI}),   is the following lower bound on the energy. 

 \begin{cor} \label{corlow}
 Suppose that assumptions of Theorem \ref{monoh} hold. Then, the following lower bound holds for the energy functional 
 \begin{equation}\label{enerLower}
  \int_{B_R} \left[ \sum_{i=1}^{m} \Phi(|\nabla u_i|^2) - 2\tilde H(u)\right] dx \ge C  R^{\alpha-n} \ \ \text{for all} \ \ R>1,
  \end{equation}
   where $C=I(1)$ is independent from $R$.  
 \end{cor}

As the next theorem, we prove an upper bound on the energy function.  

\begin{thm} \label{enopthm}
Suppose that $u=(u_i)_{i=1}^m$ is a bounded $H$-monotone solution of (\ref{main}) such that for each 
 $i=1,..,m$,
\begin{equation}
\lim_{x_n\to \infty} u_i( x',x_n)=a_i, \ \ \ \forall  x=( x', x_n)\in\mathbb{R}^{n}
\end{equation}
for some constants $a_i$.  Then 
\begin{equation}\label{energybound}
J_R(u):=\int_{B_{R}}  \left[  \sum_{i=1}^m \Phi(|\nabla u_i|^2) -2 \tilde H(u)+2\tilde H(a) \right] d x \le C R^{n-1},
\end{equation}
where  $a=(a_i)_{i=1}^{m}$ and $C$ are  independent from $R$.
 \end{thm} 
\begin{proof} Define the sequence of shift functions $u^t=(u_i^t)_{i=1}^m$  where  $u_i^t(x):=u_i(x',x_n+t)$ for $t\in\mathbb{R}$ and $x=(x',x_n)\in\mathbb R^n$. Note that $u^t=(u_i^t)_{i=1}^m$ satisfies 
\begin{equation}
\label{maintt}
-\div\left(\Phi'(|\nabla u^t_i|^2) \nabla u^t_i\right) =   H_i(u^t)   \ \ \text{in}\ \ \mathbb{R}^n. 
  \end{equation}
The fact that $u_i^t$ is convergent to $a_i$ pointwise, it is straightforward to  see that 
\begin{equation}\label{Ht}
 \int_{B_{R}} (\tilde H(u^t) -\tilde H(a)) d x \to 0 \ \ \text{when} \ \ {t\to\infty}.  
 \end{equation}

On the other hand, multiply both sides  of (\ref{maintt}) with $u_i^t-a_i$ and integrate by parts to get
\begin{eqnarray*}
\label{}
\hfill -\int_{B_{R}} \Phi'(  |\nabla u_i^t|^2) |\nabla u_i^t|^2+ \int_{\partial B_{R}} \Phi'(  |\nabla u_i^t|^2) \partial_\nu u_i^t (u_i^t-a_i) =-\int_{B_{R}}   H_i(u^t) (u_i^t-a_i),
  \end{eqnarray*}
Sending $t\to \infty$ implies that 
$$\int_{B_{R}} \Phi'(  |\nabla u_i^t|^2) |\nabla u_i^t|^2\to 0. $$
Note that due to the assumption $2s\Phi''(s)+\Phi(s)>0$ when $s>0$ and $\Phi(0)=0$ we have $0\le \Phi(s)\le 2 \Phi'(s) s$ for any $s>0$.  This implies that 
\begin{equation}\label{phit}
0\le \int_{B_{R}} \Phi(  |\nabla u_i^t|^2) \le 2\int_{B_{R}} \Phi'(  |\nabla u_i^t|^2) |\nabla u_i^t|^2\to 0 \ \ \text{as} \ \ t\to 0. 
  \end{equation}
From this and (\ref{Ht}) we get 
\begin{equation}\label{energydecayt}
\lim_{t\to\infty} J_R(u^t)=0.
\end{equation}
We now use $J_R(u^t)$ to construct an upper bound on $J_R(u)$. Note that   differentiating the energy functional with respect to $t$, one gets 
\begin{eqnarray}\label{der-energyt}
\partial_t J_R(u^t)= \sum_{i=1}^m \int_{B_{R}} \left[ 2  \Phi'(|\nabla u^t_i|^2)   \nabla u_i^t\cdot \nabla (\partial _t u_i^t)- 2  H_i(u^t) \partial_t u_i^t \right]. 
    \end{eqnarray}
Multiplying the system of equations (\ref{maintt}) with $\partial_t u^t$ and performing integration  by parts we obtain 
  \begin{eqnarray}
\label{partialtt}
\int_{B_{R}}   H_i(u^t) \partial_t u_i^t &=& \int_{B_{R}} \Phi'(|\nabla u^t_i|^2)  \nabla u_i^t\cdot \nabla (\partial _t u_i^t) 
\\&& \nonumber- \int_{\partial B_{R}}  \Phi'(|\nabla u^t_i|^2)  \partial_\nu u_i^t \partial_t u_i^t , 
  \end{eqnarray}
for each $i=1,\cdots,m$.  From (\ref{partialtt}) and  (\ref{der-energyt}) we obtain 
  \begin{eqnarray}\label{der-energy-bdt}
\partial_t E_R(u^t)= \sum_{i=1}^m  \int_{\partial B_{R}}  \Phi'(|\nabla u^t_i|^2) \partial_\nu u_i^t \partial_t u_i^t.
   \end{eqnarray}
 Note that there exist a constant $M$ such that $-M\le  \Phi'(|\nabla u^t_i|^2) \partial_\nu u^t\le M$ and $\partial_t u_i^t>0>\partial_t u_j^t$ for $i\in \bar I$ and $j\in \bar J$ . Therefore, 
  \begin{eqnarray}\label{der-energy-bdMt}
\partial_t E_R(u^t) \ge M \int_{\partial B_{R}}\left( \sum_{j\in \bar J} \partial_t u_j^t - \sum_{i\in \bar I}\partial_t u_i^t \right) d S.
   \end{eqnarray}
Therefore, 
 \begin{eqnarray}\label{der-energy-last}
 J_R(u)&=& J_R(u^t)- \int_{0}^{t} \partial_t J_R(u^s) ds 
 \nonumber\\
 &\le&J_R(u^t) +M \int_{0}^{t}  \int_{\partial B_{R}} \left(  \sum_{i\in \bar I} \partial_s u_i^s-  \sum_{j\in \bar J} \partial_s u_j^s \right) dS ds\nonumber\\
 &=&J_R(u^t) +M \int_{\partial B_{R}}  \left(  \sum_{i\in \bar I} (u_i^t-u_i)+ \sum_{j\in \bar J}(u_j-u_j^t)  \right) dS.
 \end{eqnarray}
From the definiton of $H$-monotonicity and the sets of $\bar I,\bar J$ we have  $u_i<u_i^t$ and $u_j^t<u_j$ for all $i\in \bar I$, $j\in \bar J$ and $t\in \mathbb{R^+}$. Therefore, 
\begin{equation}\label{energytt}
J_R(u) \le  J_R(u^t) +M \int_{\partial B_{R}} \left(  \sum_{i\in I} (u_i^t-u_i)+ \sum_{j\in J}(u_j-u_j^t)  \right)  dS \ \ \text{for all}\ \ t \in\mathbb{R^+},
\end{equation}  
where $M:=\max_{i=1}^m \left\{|| \Phi'(|\nabla u_i|^2) |\nabla u_i|  ||_{L^\infty(\mathbb{R}^n)}\right\}$.  The upper bound (\ref{energytt}) implies
\begin{eqnarray*}
  E_R(u) \le E_R(u^t) +C  |\partial B_R| \ \ \ \text{for all} \ \ t\in \mathbb{R^+}.
 \end{eqnarray*}
Sending  $t\to\infty$ and using  (\ref{energydecayt}), finally we obtain that  $$ J_R(u) \le C |\partial B_R|\le C R^{n-1}. $$ 
This provides the desired result. 

\end{proof}
Before we finish this section, let us mention a couple of open problems for the system of equations (\ref{main}). 
\begin{open} Under what assumptions on $H=(H_i)_{i=1}^m$ and solutions,  one can  provide a counterpart of the pointwise inequalities provided by Modica in \cite{m} and Caffarelli et al.  in \cite{cgs} for solutions of (\ref{main}) when $m \ge 2$? 
 \end{open}

\begin{open} In the light of Theorem \ref{enopthm}, Corollary \ref{corlow} and  Theorem B, one might expect that $I_\alpha(r)$ should be a nondecreasing function of $r$ when $\alpha\ge \alpha^*=1$  for the case of systems that is  when $m \ge 2$. 
\end{open}

\section{Geometric Poincar\'{e} and stability inequalities for systems}\label{secpoin}
Note that the matrix $\mathcal A(\eta):=(a_{i,j}(\eta))_{i,j=1}^n$ where $a_{i,j}(\eta)$ is defined by
\begin{equation}\label{aij1}
a_{i,j}(\eta): = 2 \Phi''(|\eta|^2) \eta_i\eta_j+ \Phi'(|\eta|^2) \delta_{ij}. 
\end{equation}
 is symmetric and positive definite for every $\eta\in\mathbb R^n$. This is because for any $\zeta\in\mathbb R^n$, 
 \begin{eqnarray}
\label{Aetazeta} \mathcal A(\eta) \zeta.\zeta &=&\sum_{i,j=1}^m 2 \zeta_i \zeta_j \Phi''(|\eta|^2) \eta_i\eta_j +\Phi'(|\eta|^2) \zeta_i\zeta_j \delta_{i,j}
\\&=& \nonumber 2\label{phi''}  \Phi''(|\eta|^2) |\zeta\cdot \eta|^2 +\Phi'(|\eta|^2) |\zeta|^2.
\end{eqnarray}
Note that when $\Phi''(\eta)$ is positive clearly $A(\eta) \zeta.\zeta $ is positive since $\Phi'$ is positive and when $\Phi''(\eta)$  is negative applying Young's inequality together with $2\Phi''(s)s+\Phi'(s)>0$ when $s>0$ implies that $A(\eta) \zeta.\zeta $ is positive. We are now ready to prove the stability inequality for solutions of (\ref{main}). Note that such an inequality for the case of semilinear systems is given in \cite{cf,fg,mf}.

\begin{lemma}\label{stabineqphi}  
  Let $u=(u_i)_{i=1}^m$ denote a stable solution of (\ref{main}).  Then 
\begin{equation} \label{stability}
\sum_{i,j=1}^{m} \int_{\mathbb R^n}  \sqrt{\partial_j H_i(u) \partial_i H_j(u)} \zeta_i \zeta_j \le \sum_{i=1}^{m} \int_{\mathbb R^n}  \mathcal A(\nabla u_i) \nabla \zeta_i \cdot \nabla \zeta_i, 
\end{equation} 
for any $\zeta=(\zeta_i)_{i=1}^m$ where $ \zeta_i \in C_c^1(\mathbb R^n)$ for $1\le i\le m$. 
\end{lemma}  
\begin{proof}   Since $u$ is a stable solutions,  there exists a sequence $\phi=(\phi_i)_i^m$ that satisfies (\ref{L}).  Consider a test function  $\zeta=(\zeta_i)_i^m$ where $ \zeta_i \in L^{\infty} (\mathbb R^n) \cap H^1(\mathbb R^n)$ with compact support  and multiply both sides of (\ref{L}) with  $\frac{\zeta_i^2}{\phi_i}$. Integrating by parts we get 
\begin{equation}\label{IiJi}
 I_i:= \sum_{j=1}^{m}  \int_{\mathbb R^n}  \partial_j H_i(u) \phi_j \frac{\zeta_i^2}{\phi_i} = \int_{\mathbb R^n}  \mathcal A(\nabla u_i) \nabla \phi_i \cdot  \left(2 \nabla \zeta_i \frac{\zeta_i}{\phi_i} - \nabla \phi_i \frac{\zeta^2_i}{\phi_i^2}\right) =:J_i.
  \end{equation}
The fact that $\mathcal A(\nabla u_i)$ is positive definite we get 
\begin{eqnarray*}
0&\le& \mathcal A(\nabla u_i) (\phi_i \nabla \zeta_i- \zeta_i \nabla \phi_i)\cdot (\phi_i \nabla \zeta_i- \zeta_i \nabla \phi_i)
\\&=& \phi_i^2 \mathcal A(\nabla u_i) \nabla \zeta_i\cdot \nabla \zeta_i + \zeta_i^2 \mathcal A(\nabla u_i) \nabla \phi_i \cdot \nabla \phi_i - \zeta_i \phi_i \mathcal A(\nabla u_i)  \nabla \phi_i \cdot \nabla \zeta_i. 
\end{eqnarray*}
Applying  this to (\ref{IiJi}) for each $i$ we obtain
 \begin{equation}\label{stJ}
 J_i \le \int_{\mathbb R^n}  \mathcal A(\nabla u_i) \nabla \zeta_i \cdot  \nabla \zeta_i. 
 \end{equation}
 For the left-hand side we have,
 \begin{eqnarray*}
  \sum_{i=1}^{m} I_i&=&  \sum_{i,j=1}^{m}   \int_{\mathbb R^{n}} \partial_j H_i(u) \phi_j \frac{\zeta_i^2}{\phi_i}
\\& =&   \sum_{i<j}^{m}   \int_{\mathbb R^{n}} \partial_j H_i(u) \phi_j \frac{\zeta_i^2}{\phi_i} + \sum_{i>j}^{n}   \int_{\mathbb R^{n}} \partial_j H_i(u) \phi_j \frac{\zeta_i^2}{\phi_i} +  \sum_{i=1}^{m}  \int_{\mathbb R^{n}} \partial_i H_i(u) {\zeta_i^2}
\\&=&
  \sum_{i<j}^{m}   \int_{\mathbb R^{n}} \partial_j H_i(u) \phi_j \frac{\zeta_i^2}{\phi_i} + \sum_{i<j}^{m}   \int_{\mathbb R^{n}} \partial_i H_j(u) \phi_i \frac{\zeta_j^2}{\phi_j} + \sum_{i=1}^{m}  \int_{\mathbb R^{n}}  \partial_i H_i(u) {\zeta_i^2}
  \\&=& \sum_{i<j}^{m}   \int_{\mathbb R^{n}}  \left( \partial_j H_i(u) \phi_j \frac{\zeta_i^2}{\phi_i} + \partial_i H_j(u) \phi_i \frac{\zeta_j^2}{\phi_j}  \right)+\sum_{i=1}^{m}   \int_{\mathbb R^{n}} \partial_i H_i(u) {\zeta_i^2}
  \\&\ge & 2 \sum_{i<j}^{m}   \int_{\mathbb R^{n}}   \sqrt{\partial_j H_i(u) \partial_i H_j(u)} \zeta_i \zeta_j+  \sum_{i=1}^{m}  \int_{\mathbb R^{n}} \partial_i H_i(u) {\zeta_i^2}
  \\&=&\sum_{i,j=1}^{m}   \int_{\mathbb R^{n}} \sqrt{\partial_j H_i(u) \partial_i H_j(u)} \zeta_i\zeta_j.
  \end{eqnarray*}
This finishes the proof. 

\end{proof}  

We now apply the stability inequality to provide a geometric Poincar\'{e} inequality of the following from. For the case of scalar equations that is when $m=1$ this inequality was driven by Sternberg-Zumbrun in \cite{sz} and it was applied in this context by  Farina-Sciunzi-Valdinoci \cite{fsv} and references therein to provide De Giorgi type results.   Note also that Cabr\'{e} applied this type inequality to prove regularity of extremal solutions of nonlinear eigenvalue problems in \cite{cab}. For the case of system of equations that is when $m\ge 1$ this inequality was first proved by Ghoussoub and the author in \cite{fg} and they applied the inequality to conclude De Giorgi type results for system of equations. Let us mention that interested readers can find similar geometric Poincar\'{e} inequalities in these references as well  \cite{f2,dp,d,sv}.

\begin{thm}\label{lempoin}
 Assume that  $m,n\ge 1$ and $ u=(u_i)_{i=1}^m$ is a stable solution of (\ref{main}).  Then, for any $\eta=(\eta_k)_{k=1}^m \in C_c^1(\mathbb R^n)$, the following inequality holds;
\begin{eqnarray}\label{poincare}
&&
\sum_{i\neq j} \int_{\mathbb R^{n}}   \left[  \sqrt{\partial_{j} H_i( u) \partial_{ i} H_j( u) }  |\nabla  u_i|  |\nabla  u_j| \eta_i \eta_j   - \partial_{j}H_i(u)  \nabla  u_i \cdot   \nabla  u_j \eta_i^2 \right] 
\\&& \nonumber+
\sum_{i=1}^m    \int_{    \{|\nabla u_i|\neq 0 \} \cap \mathbb R^n }  \Phi'(|\nabla u_i|^2)   |\nabla u_i|^2 \mathcal{K}_i^2 \eta_i^2 
\\&& \nonumber +\sum_{i=1}^m    \int_{    \{|\nabla u_i|\neq 0 \} \cap \mathbb R^n } \left[  2 \Phi''(|\nabla u_i|^2) |\nabla u_i^2|  +\Phi'(|\nabla u_i|^2)  \right] | \nabla_{T_i} |\nabla  u_i| |^2   \eta_i^2
\\&\le& \nonumber\sum_{i=1}^{m}  \int_{\mathbb R^{n}}  2 |\nabla u_i|^2 \Phi''(|\nabla u_i|^2) |\nabla u_i\cdot \nabla\eta_i|^2 +\Phi'(|\nabla u_i|^2) |\nabla u_i|^2 |\nabla \eta_i|^2,
  \end{eqnarray} 
 % $$\mathcal A(\nabla u_i) \nabla \eta_i \cdot \nabla \eta_i ,$$
 where $\nabla_{T_i}$ stands for the tangential gradient along a given level set of $u_i$ and 
$\mathcal{K}_i^2$ for the sum of  squares of  principal curvatures of such a level set.
\end{thm}

\begin{proof} Suppose that $\eta=(\eta_1,...,\eta_m)$ for $\eta_i\in C^1_c(\mathbb R^n)$ is a test function.  Test the stability inequality (\ref{stability}) with $\zeta_i=|\nabla u_i|\eta_i$  to get 
 \begin{eqnarray}\label{stablepoin}
   \sum_{i=1}^m \int_{\mathbb R^n}  \partial_i H_i(u) |\nabla u_i|^2\eta_i^2 &\le& -\sum_{i\neq j} \int_{\mathbb R^n} \sqrt{\partial_j H_i(u) \partial_i H_j(u)} |\nabla u_i| |\nabla u_j|\eta_i\eta_j 
   \\&&+\nonumber  \sum_{i=1}^m \int_{\mathbb R^n}|\nabla u_i|^2 \mathcal A(\nabla u_i)  \nabla \eta_i \cdot \nabla \eta_i 
\\&&\nonumber +  \sum_{i=1}^m \int_{\mathbb R^n} \eta_i^2 \mathcal A(\nabla u_i) \nabla |\nabla u_i|\cdot \nabla |\nabla u_i| 
\\&&\nonumber
+  \frac{1}{2}  \sum_{i=1}^m \int_{\mathbb R^n}   \mathcal A(\nabla u_i) \nabla |\nabla u_i|^2 \cdot \nabla \eta_i^2 .
  \end{eqnarray} 
  Straightforward calculations show that for each $k$,  
  \begin{equation}\label{partialk}
\partial_{x_k} \left( \Phi'(|\nabla u_i|^2) \nabla u_i \right)= \mathcal A(\nabla u_i) \nabla \partial_{x_k} u_i.  
  \end{equation}
  Applying this and differentiating the $i^{th}$ equation of (\ref{main}) with respect to $x_k$ for each $i=1,2,...,m$ we get 
  \begin{equation}\label{pdepartialk}
-\div\left(  \mathcal A(\nabla u_i) \nabla \partial_{x_k} u_i   \right) = \sum_{j=1}^m \partial_j H_i(u) \partial_{x_k} u_i .
    \end{equation}
Multiplying the above with $\eta_i^2 \partial_k u_i $,  integrating by parts and taking sum on the indices $i,k$ we get 
 \begin{eqnarray}\label{iHi}
&& \sum_{i=1}^m  \int_{\mathbb R^n}  \partial_i H_i(u) |\nabla u_i|^2\eta_i^2
 \\&=& \nonumber -\sum_{j\neq i} \int_{\mathbb R^n} \partial_j H_i(u) \nabla u_i\cdot \nabla u_j \eta_i^2   + \sum_{i=1}^m \sum_{k=1}^n  \int_{\mathbb R^n}  \mathcal A(\nabla u_i) 
  \nabla (\partial_{x_k} u_i)\cdot \nabla (\partial_{x_k} \eta_i^2)
  \\&=& \nonumber -\sum_{j\neq i} \int_{\mathbb R^n} \partial_j H_i(u) \nabla u_i\cdot \nabla u_j \eta_i^2 +\sum_{i=1}^m \sum_{k=1}^n  \int_{\mathbb R^n} \eta_i^2 \mathcal A(\nabla u_i) 
  \nabla (\partial_{x_k} u_i)\cdot \nabla (\partial_{x_k} u_i)  
  \\&&\nonumber +  \frac {1}{2} \sum_{i=1}^m  \int_{\mathbb R^n} \mathcal A(\nabla u_i) 
  \nabla |\nabla u_i|^2 \cdot \nabla  \eta_i^2.
   \end{eqnarray} 
  Equating (\ref{iHi}) and (\ref{stablepoin}) we get the following since the term $ \frac 1 2 \sum_{k=1}^m  \int_{\mathbb R^n} \mathcal A(\nabla u_i) \nabla |\nabla u_i|^2 \cdot \nabla  \eta_i^2$ cancels out,  
   \begin{eqnarray}\label{iHij}
&&\sum_{i\neq j} \int_{\mathbb R^n}\left[ \sqrt{\partial_j H_i(u) \partial_i H_j(u)} |\nabla u_i| |\nabla u_j|\eta_i\eta_j -  \partial_j H_i(u) \nabla u_i\cdot \nabla u_j \eta_i^2\right]
\\&&  \label{etakk}+  \sum_{i,k=1}^m \int_{\mathbb R^n}  \mathcal A(\nabla u_i) 
  \nabla (\partial_{x_k} u_i)\cdot \nabla (\partial_{x_k} u_i) \eta_i^2
  \\&& \label{etanab} -  \sum_{i=1}^m \int_{\mathbb R^n}  \mathcal A(\nabla u_i) \nabla |\nabla u_i|\cdot \nabla |\nabla u_i|  \eta_i^2
  \\&\le&  \sum_{i=1}^m \int_{\mathbb R^n}|\nabla u_i|^2 \mathcal A(\nabla u_i)  \nabla \eta_i \cdot \nabla \eta_i .
      \end{eqnarray} 
For the rest of the proof, we simplify two terms (\ref{etakk}) and (\ref{etanab}) in the left-hand side of the above inequality.   From the definition of there matrix $\mathcal A$, in the light of (\ref{Aetazeta}), we get 
   \begin{eqnarray}\label{Akk}
\sum_{k=1}^n  \mathcal A(\nabla u_i) \nabla (\partial_{x_k} u_i)\cdot \nabla (\partial_{x_k} u_i)
  &=& 2 \Phi''(|\nabla u_i|^2)   \sum_{k=1}^n | \nabla u_i\cdot \nabla \partial_{x_k} u_i |^2
  \\&& \nonumber + \Phi'(|\nabla u_i|^2) \sum_{k=1}^n | \nabla \partial_{x_k} u_i  |^2. 
         \end{eqnarray} 
Straightforward calculations show that 
   \begin{equation}\label{Akk1}
\sum_{k=1}^n | \nabla u_i\cdot \nabla \partial_{x_k} u_i  |^2= \frac{1}{4} \left|  \nabla |\nabla u_i|^2  \right|^2 =|\nabla u_i|^2 \left|  \nabla |\nabla u_i|  \right|^2 . 
         \end{equation} 
From (\ref{Akk}) and (\ref{Akk1}) we obtain the following form for the term in (\ref{etakk}) 
 \begin{eqnarray}\label{Akk2}
&& \nonumber \sum_{i=1}^m \sum_{k=1}^n \int_{\mathbb R^n}   \mathcal A(\nabla u_i) \nabla (\partial_{x_k} u_i)\cdot \nabla (\partial_{x_k} u_i) \eta_i^2
 \\ &=&  \sum_{i=1}^m \int_{\mathbb R^n}   2 |\nabla u_i|^2  \Phi''(|\nabla u_i|^2)  \left|  \nabla |\nabla u_i|  \right|^2 \eta_i^2
  \\&& \nonumber + \sum_{i=1}^m  \int_{\mathbb R^n}   \Phi'(|\nabla u_i|^2) \sum_{k=1}^n | \nabla \partial_{x_k} u_i  |^2 \eta_i^2. 
         \end{eqnarray} 
Similarly, from the definition of the matrix $\mathcal A$, i.e. using (\ref{Aetazeta}), we get 
\begin{eqnarray}\label{Auu}
\nonumber  \mathcal A(\nabla u_i) \nabla |\nabla u_i|\cdot \nabla |\nabla u_i| &=& 2 \Phi''(|\nabla u_i|^2) \left| \nabla u_i\cdot  \nabla |\nabla u_i| \right|^2
\\&&+ \Phi'(|\nabla u_i|^2) \left| \nabla |\nabla u_i| \right|^2.
\end{eqnarray}
This implies that the term in (\ref{etanab}) is of the from 
\begin{eqnarray}\label{Auu1}
&& \nonumber \sum_{i=1}^m \int_{\mathbb R^n}  \mathcal A(\nabla u_i) \nabla |\nabla u_i|\cdot \nabla |\nabla u_i| \eta_i^2 
\\&=& \sum_{i=1}^m \int_{\mathbb R^n} 2 \Phi''(|\nabla u_i|^2) \left| \nabla u_i\cdot  \nabla |\nabla u_i| \right|^2 \eta_i^2 
\\&&\nonumber +\sum_{i=1}^m \int_{\mathbb R^n} \Phi'(|\nabla u_i|^2) \left| \nabla |\nabla u_i| \right|^2 \eta_i^2.
\end{eqnarray}
The difference of (\ref{Akk2}) and (\ref{Akk}), as appeared in (\ref{etanab}) and (\ref{etakk}), is  
\begin{eqnarray}\label{Auu2}
&&\sum_{i=1}^m \sum_{k=1}^n  \int_{\mathbb R^n}   \mathcal A(\nabla u_i) \nabla (\partial_{x_k} u_i)\cdot \nabla (\partial_{x_k} u_i) \eta_i^2
\\&&-\sum_{i=1}^m \int_{\mathbb R^n}  \mathcal A(\nabla u_i) \nabla |\nabla u_i|\cdot \nabla |\nabla u_i| \eta_i^2 
\\&=& \label{uT}  2 \sum_{i=1}^m  \int_{   \Omega  } \eta_i^2   |\nabla u_i|^2 
 \left[ \left|  \nabla |\nabla u_i|  \right|^2- \frac{1}{|\nabla u_i|^2}  \left| \nabla u_i\cdot  \nabla |\nabla u_i| \right|^2\right]   \Phi''(|\nabla u_i|^2) 
\\&&\label{uT1}+ \sum_{i=1}^m  \int_{   \Omega  } \Phi'(|\nabla u_i|^2) \left[ \sum_{k=1}^n | \nabla \partial_{x_k} u_i  |^2 - \left| \nabla |\nabla u_i| \right|^2 \right] \eta_i^2 .
\end{eqnarray}
where $\Omega=\{|\nabla u_i|\neq 0 \} \cap \mathbb R^n$.  We now simplify (\ref{uT}) and (\ref{uT1}) via applying the tangential gradient and curvatures. Suppose that $|\nabla u_i|\neq 0$ at a point $x\in\mathbb R^n$, then 
\begin{equation}\label{levelset1}
|\nabla_T |\nabla u_i||^2= \left|  \nabla |\nabla u_i|  \right|^2- \frac{1}{|\nabla u_i|^2}  \left| \nabla u_i\cdot  \nabla |\nabla u_i| \right|^2,
\end{equation}
 where $\nabla_T$ denotes the orthogonal projection of the gradient along this level set. In addition, according to formula (2.1) given in  \cite{sz}, the following geometric identity between the tangential gradients and curvatures holds, 
 \begin{equation}\label{levelset2}
 \sum_{k=1}^{n} |\nabla \partial_k u_i|^2-|\nabla|\nabla u_i||^2= |\nabla u_i|^2 \mathcal K^2_i+|\nabla_T|\nabla u_i||^2 ,
 \end{equation}
 for $\mathcal K^2_i:= \sum_{l=1}^{n-1} \mathcal{\kappa}_l^2 $ where $ \mathcal{\kappa}_l$ is the principal curvatures of the level set of $u_i$ at $x$.  Substituting (\ref{levelset1}) and (\ref{levelset2}) in (\ref{uT}) and (\ref{uT1}) we get 
 \begin{eqnarray}\label{Auu3}
&&\sum_{i=1}^m \sum_{k=1}^n  \int_{\mathbb R^n}   \mathcal A(\nabla u_i) \nabla (\partial_{x_k} u_i)\cdot \nabla (\partial_{x_k} u_i) \eta_i^2
\\&&\nonumber  -\sum_{i=1}^m \int_{\mathbb R^n}  \mathcal A(\nabla u_i) \nabla |\nabla u_i|\cdot \nabla |\nabla u_i| \eta_i^2 
\\&=& \nonumber \sum_{i=1}^m  \int_{    \{|\nabla u_i|\neq 0 \} \cap \mathbb R^n }  2 |\nabla u_i|^2 \Phi''(|\nabla u_i|^2) |\nabla_T |\nabla u_i||^2 \eta_i^2 
\\&&\nonumber + \sum_{i=1}^m  \int_{    \{|\nabla u_i|\neq 0 \} \cap \mathbb R^n } \Phi'(|\nabla u_i|^2) \left[ |\nabla u_i|^2 +|\nabla_T|\nabla u_i||^2 \right] \eta_i^2 .
\end{eqnarray}
 Finally, substitution of (\ref{Auu3}) in (\ref{etanab}) and (\ref{etakk}) completes the proof.    
 
\end{proof}

\section{De Giorgi type results for symmetric systems} \label{secde}
In this section, we provide One dimensional symmetry results for stable and $H$-monotone solutions of symmetric system (\ref{main}) in lower dimensions with a general nonlinearity.     At first let us fix a few notations.  Throughout this section we suppose that   $\zeta=(\zeta_i)_{i=1}^m$ is a sequence of test functions where $\zeta_i\in C^2_c(\mathbb R^n)\cap[0,1]$ where $\zeta_i\equiv1$ in $B_1$ and $\zeta_i \equiv 0$ in $\mathbb R^n \setminus B_2$. Set ${}_R \zeta_{i}(x):=\zeta_i(\frac{x}{R})$ and 
${}_R \Gamma=({}_R\Gamma_i)_{i=1}^m$  where ${}_R \Gamma_i(x):=\nabla \zeta_i(\frac{x}{R}) $ for any $R>1$.  Note that $||{}_R \Gamma_i(x)||_{L^\infty(\mathbb R^n)}\le C$ where $C$ is independent  from $R$ and  $\nabla {}_R\zeta_i(x)=R^{-1}{}_R\Gamma_i (x)$.  

To set up a Liouville theorem for the quotient of partial derivatives of solutions of (\ref{main}), we first state the following technical lemma.

\begin{lemma}\label{lmlinear}
Suppose that $u=(u_i)_{i=1}^m$ is a $H$-monotone solution of (\ref{main}). Set $\phi_i:=\partial_{x_n} u_i$ and $\psi_i:= \nabla u_i\cdot \eta$ where $\eta=(\eta',0)\in\mathbb R^{n-1}\times \{0\}$ and define $\sigma_i:=\frac{\psi_i}{\phi_i}$. Then the sequence of functions $\sigma=(\sigma_i)_{i=1}^m$ satisfies 
\begin{eqnarray}\label{insigma1}
&& \div\left[  \phi_i^2 \mathcal A(\nabla u_i) \nabla \sigma_i \right] 
+\sum_{j=1}^{m} \partial_j H_{i} (u) \phi_i \phi_j (\sigma_j-\sigma_i)\sigma_i = 0   \ \ \text{in}\ \ \mathbb {R}^n.
  \end{eqnarray}
\end{lemma}
\begin{proof}
Since the proof is straightforward we omit it here. 
\end{proof}

The fact that $\sigma=(\sigma_i)_{i=1}^m$ satisfies  (\ref{insigma1}) motivates us to provide a Liouville theorem for system (\ref{insigma1}). Applying Caccioppoli type arguments we establish the following Liouvlle theorem for a slightly more general setting than (\ref{insigma1}).  Let us mention that for the case of scalar semilinear equation, $m=1$ and $\Phi(s)=s$, this type of Liouville theorem was noted by Berestycki, Caffarelli and Nirenberg in \cite{bcn} and used by Ghoussoub and Gui in \cite{gg1} and later by Ambrosio and Cabr\'{e} in \cite{ac}  to prove the De Giorgi conjecture in dimensions two and three.    Also, Ghoussoub and Gui in \cite{gg2} used a slightly stronger version to show that the De Giorgi's conjecture is true in dimensions four and five for a special class of solutions that satisfy an antisymmetry condition. We also refer interested readers to \cite{big} by Barlow,  Bass and  Gui  and to \cite{bar} by Barlow for some probability based arguments regarding this Liouvlle theorem.  

For the case of scalar quasilinear equation, $m=1$ and a general $\Phi$, this Liouville theorem is provided by  Farina, Sciunzi and Valdinoci in \cite{fsv} and by Danielli and  Garofalo  in  \cite{dg}. For the case of semilinear system of equations, $m\ge1$ and $\Phi(s)=s$, we refer to \cite{fg} by Ghoussoub and the author.

\begin{prop}\label{liouville} Assume that  for each  $i=1,\cdots,m$ functions $|\nabla u_i|$ and $\phi_i$ are locally bounded in $\mathbb{R}^n$ where $\phi^2_i >0 $ and $\sigma_i \in H^1_{loc}(\mathbb{R}^n)$.  Let  
\begin{equation}\label{liouassum}
\sum_{i=1}^{m}\int_{B_{2R}\setminus B_R}\phi_i^2\sigma_i^2
\mathcal A(\nabla u_i) {}_R\Gamma_i \cdot {}_R\Gamma_i \le C R^2 ,  
\end{equation}
 where the constant $C$ is independent from $R>1$.  Let $\sigma=(\sigma_i)_{i=1}^m$  satisfy
\begin{eqnarray}\label{insigma}
\label{div}
 \sigma_i \div\left[  \phi_i^2 \mathcal A(\nabla u_i) \nabla \sigma_i \right]  +\sum_{j=1}^{m} h_{ij}(x) f(\sigma_j-\sigma_i)\sigma_i \ge 0   \ \ \text{in}\ \ \mathbb {R}^n,
  \end{eqnarray}
where $0\le h_{ij}\in L_{loc}^1(\mathbb{R}^n)$, $h_{ij}=h_{ji}$ and $f\in L_{loc}^1(\mathbb{R})$ is an odd function such that $f(s)\ge 0$ for $s\in\mathbb R^+$.  Then,  each function $\sigma_i$ is constant for all  $i=1,...,m$.
\end{prop}
%%%%%%%%
\begin{proof}  The proof is strongly motived by the ideas and methods used in \cite{fg,mf,ac,aac,fsv}. Note that 
\begin{eqnarray*}
\sum_{i,j=1}^m h_{ij}(x) \sigma_i f(\sigma_j-\sigma_i)&=& \sum_{i< j}  h_{ij} \sigma_i f(\sigma_j-\sigma_i) + \sum_{i> j} h_{ij}  \sigma_i  f(\sigma_j-\sigma_i)  \\&=&\sum_{i< j} h_{ij}  \sigma_i f(\sigma_j-\sigma_i)  + \sum_{i< j}h_{ij} \sigma_j f(\sigma_i-\sigma_j)   \ \ \text{since} \ \ h_{ij}=h_{ji}
\\&=&-\sum_{i< j} h_{ij} (\sigma_j-\sigma_i)f(\sigma_j-\sigma_i)  \ \ \text{since $f$ is odd}.
\\&\le &0 \ \ \text{since $h_{ij} \ge 0$ and $sf(s)\ge 0$ for any $s\in \mathbb R$}.
  \end{eqnarray*}
Multiply (\ref{insigma})  with a text function ${{}_R\zeta_i}^2$  and perform integration  by parts to get 
\begin{eqnarray*}
 \sum_{i=1}^{m}\int_{\mathbb R^n} \phi_i^2
\mathcal A(\nabla u_i) \nabla (\sigma_i {}_R\zeta_i^2 ) \cdot \nabla \sigma_i \le  \sum_{i,j=1}^m \int_{\mathbb R^n} h_{ij}(x) \sigma_i f(\sigma_j-\sigma_i) \le  0.
  \end{eqnarray*}
Therefore, 
\begin{equation}\label{phizeta}
\sum_{i=1}^{m}\int_{\mathbb R^n} \phi_i^2 {}_R\zeta_i^2
\mathcal A(\nabla u_i) \nabla  \sigma_i \cdot \nabla \sigma_i \le -2 \sum_{i=1}^{m} \int_{B_{2R}\setminus B_R} \phi_i^2  \sigma_i {}_R\zeta_i    \mathcal A(\nabla u_i) \nabla {}_R \zeta_i  \cdot \nabla \sigma_i.
  \end{equation}
In the light of the Cauchy inequality with epsilon, one can see that  for any $\epsilon>0$,  
$$ 2|\sigma_i| {}_R\zeta_i    |\mathcal A(\nabla u_i) \nabla  {}_R\zeta_i  \cdot \nabla \sigma_i | \le  \epsilon {}_R \zeta_i^2  \mathcal A(\nabla u_i) \nabla  \sigma_i  \cdot \nabla \sigma_i +  \frac{1}{\epsilon} \sigma_i^2  \mathcal A(\nabla u_i) \nabla  {}_R\zeta_i  \cdot \nabla {}_R\zeta_i .$$
From this and (\ref{phizeta}), for every $R>1$, we get  
\begin{eqnarray}\label{gggg}
\ \ \ \ \ \ \sum_{i=1}^{m}\int_{\mathbb R^n} \phi_i^2 {}_R\zeta_i^2
\mathcal A(\nabla u_i) \nabla  \sigma_i \cdot \nabla \sigma_i &\le& \epsilon \sum_{i=1}^{m} \int_{B_{2R}\setminus B_R} \phi_i^2 {}_R\zeta_i^2
\mathcal A(\nabla u_i) \nabla  \sigma_i \cdot \nabla \sigma_i 
\\&& \nonumber + \frac 1 \epsilon 
\sum_{i=1}^{m} \int_{B_{2R}\setminus B_R} \phi_i^2 
\mathcal A(\nabla u_i) \nabla {}_R\zeta_i \cdot \nabla {}_R\zeta_i.
 \end{eqnarray}
From this, (\ref{liouassum}) and the definition of the test function ${}_R\zeta_i$ we conclude that   the following integral is bounded,
\begin{equation}\label{bound1}
\sum_{i=1}^{m}\int_{\mathbb R^n} \phi_i^2 
\mathcal A(\nabla u_i) \nabla  \sigma_i \cdot \nabla \sigma_i<\infty.
\end{equation}
Now, sending $R\to \infty $ and $\epsilon \to \infty$ and applying (\ref{gggg}) and (\ref{liouassum}) show  that the integral in (\ref{bound1}) vanishes for every $i=1,\cdots,m$.  Finally, the fact that  $\mathcal A$ is a positive definite matrix implies  $|\nabla \sigma_i |\equiv 0$ for each $i=1,\cdots,m$. This completes the proof.  

\end{proof}

We are now ready to present the following De Giorgi type results in two dimensions. 

\begin{thm}\label{thm2} Suppose that $u=(u_i)_{i=1}^m$ is a classical bounded stable solution of symmetric system (\ref{main}) in two dimensions.   Assume also that $|\nabla u_i| \in L^\infty(\mathbb R^2) \cap W^{1,2}_{loc}(\mathbb R^2)$.   Then each $u_i$ is a one dimensional function, i.e. there exists $u_i^*:\mathbb R\to \mathbb R$ and $a\in \mathbb S^1$ such that $u_i(x)=u_i^*(a\cdot x)$.  In addition, the angle between $\nabla u_i$ and $\nabla u_j$ is $\arccos\left(  \frac{|\partial_i H_j|}{\partial_i H_j} \right)$. 
\end{thm}
\begin{proof} We apply the geometric Poincar\'{e} inequality given as Theorem \ref{lempoin} to provide a proof.   Ideas and method applied in this proof are strongly motivated by the ones given for the case of the scalar  equation by Berestycki, Caffarelli and Nirenberg in \cite{bcn}, Ghoussoub and Gui in \cite{gg1}, Farina,  Sciunzi and  Valdinoci in \cite{fsv} and references therein. In addition, for the case of system of equations we refer interested readers to \cite{fg,fs} by Ghoussoub,  Sire and  the author.      %We provide the proof is a few steps.  % First of all we claim that there exists a constant $C$ such that 
%\begin{equation} \int_{B_R\setminus B_{\sqrt R}} \frac{\Phi'(|\nabla u_i|^2) |\nabla u_i|^2}{|x|^2} \le C\ln R, \end{equation} 
Note that from boundedness of $|\nabla u_i|$ and  $\Phi'\in C(\mathbb R^+)$ in two dimensions we have 
\begin{equation}\label{mm}
\int_{B_R} \Phi'(|\nabla u_i|^2) |\nabla u_i|^2  \le   C R^2,
 \end{equation}
for any $R>1$. This can be also proved by multiplying (\ref{main}) by ${}_R \zeta_i^2 u_i$ and integrating by parts.    Straightforward calculations show that for each $i$ we have  
\begin{eqnarray*}
&&\int_{B_{R}\setminus B_{\sqrt R}}  \frac{1}{ | x|^2} \Phi'(|\nabla u_i|^2) |\nabla u_i|^2 d  x
\\&=& 2\int_{B_{R}\setminus B_{\sqrt R}}  \int_{| x|}^{R} \tau^{-3} \Phi'(|\nabla u_i|^2) |\nabla u_i|^2 d\tau d  x + \frac{1}{R^2}  \int_{B_{R}\setminus B_{\sqrt R}} \Phi'(|\nabla u_i|^2) |\nabla u_i|^2 
\\&\le & 2 \int_{\sqrt R}^{R}  \tau^{-3}  \int_{B_{\tau}} \Phi'(|\nabla u_i|^2) |\nabla u_i|^2  d  x d\tau + \frac{1}{ R^2}  \int_{B_{R}} \Phi'(|\nabla u_i|^2) |\nabla u_i|^2,
\end{eqnarray*}
where we have used the Fubini's theorem.  From this and (\ref{mm}) we get 
\begin{equation}\label{logR}
\int_{B_{R}\setminus B_{\sqrt R}}  \frac{\Phi'(|\nabla u_i|^2) |\nabla u_i|^2}{| x|^2} 
\le C \log R. 
\end{equation}
Now for each $i$ set $\eta_i$ to be the following standard  test function 
$$\eta_i (x):=\left\{
                      \begin{array}{ll}
                        \frac{1}{2}, & \hbox{if $|x|\le\sqrt{R}$,} \\
                      \frac{ \log R-\log |x|}{\log R}, & \hbox{if $\sqrt{R}< |x|< R$,} \\
                       0, & \hbox{if $|x|\ge R$.}
                                                                       \end{array}
                    \right.$$
Note that $ \nabla \eta_i(x)=-\frac{x}{|x|^2\log R}$ on $\sqrt{R}< |x|< R$. Therefore, the right-hand side of the inequality given in  theorem \ref{lempoin} is of the form 
\begin{eqnarray*}
&&\sum_{i=1}^{m}  \int_{\mathbb R^{n}}  2 |\nabla u_i|^2 \Phi''(|\nabla u_i|^2) |\nabla u_i\cdot \nabla\eta_i|^2 +\Phi'(|\nabla u_i|^2) |\nabla u_i|^2 |\nabla \eta_i|^2
\\&=&  \sum_{i=1}^{m}  \int_{B_{R}\setminus B_{\sqrt R}}  |\nabla u_i|^2 \mathcal A(\nabla u_i) \nabla \eta_i\cdot \nabla \eta_i
\\&\le& \frac{C}{\log^2 R } \sum_{i=1}^{m} \int_{B_{R}\setminus B_{\sqrt R}} \frac{1}{|x^4|} |\nabla u_i|^2 \mathcal A(\nabla u_i)  x\cdot x
\\&\le&  \frac{C}{\log^2 R }\sum_{i=1}^{m} \int_{B_{R}\setminus B_{\sqrt R}}  \frac{\Phi'(|\nabla u_i|^2) |\nabla u_i|^2}{| x|^2} 
\\&\le& \frac{C}{\log R}, 
\end{eqnarray*}
where we have used  (\ref{logR}) to conclude the last inequality. From this and the geometric inequality given by (\ref{poincare}) and the fact that the system is symmetric we get 
 \begin{eqnarray}\label{}
&&
\sum_{i\neq j} \int_{B_{\sqrt R}}   \left( |\partial_{j} H_i( u)|   |\nabla  u_i|  |\nabla  u_j|  - \partial_{j}H_i(u)  \nabla  u_i \cdot   \nabla  u_j  \right) 
\\&& \nonumber+
\sum_{i=1}^m    \int_{    \{|\nabla u_i|\neq 0 \} \cap B_{\sqrt R} }  \Phi'(|\nabla u_i|^2)   |\nabla u_i|^2 \mathcal{K}_i^2 
\\&& \nonumber +\sum_{i=1}^m    \int_{    \{|\nabla u_i|\neq 0 \} \cap B_{\sqrt R} } \left[  2 \Phi''(|\nabla u_i|^2) |\nabla u_i^2|  +\Phi'(|\nabla u_i|^2)  \right] | \nabla_{T_i} |\nabla  u_i| |^2   
\\&\le&\nonumber \frac{C}{\log R}. 
  \end{eqnarray} 
Now sending $R\to\infty$ and the fact that all of the terms in the left-hand side are nonnegative imply that each $u_i$ is one dimensional function and  $|\partial_{j} H_i( u)|   |\nabla  u_i|  |\nabla  u_j|  = \partial_{j}H_i(u)  \nabla  u_i \cdot   \nabla  u_j$.  The latter implies that the angle between $\nabla u_i$ and $\nabla u_j$ is precisely $\arccos \left(\frac{|\partial_{j} H_i( u)| }{\partial_{j}H_i(u)}\right)$ when $i\neq j$. This completes the proof.  
\end{proof}

Note that in the statement of Theorem \ref{thm2}, $\Phi$ does not have to satisfy conditions (A) or (B). However in the next theorem that is in regards to three dimensions one of conditions (A) or (B) is needed.   For the case of semilinear systems in two dimensions we refer interested readers to \cite{abg} for the  construction of two dimensional solutions  in the absence of stability and $H$-monotonicity,

\begin{thm}\label{thm3} Suppose that $u=(u_i)_{i=1}^m$ is a classical bounded $H$-monotone solution of symmetric system (\ref{main}) in three dimensions. Let $\Phi$ satisfy one of conditions (A) or (B).   Assume also that $|\nabla u_i| \in L^\infty(\mathbb R^3) \cap W^{1,2}_{loc}(\mathbb R^3)$.   Then each $u_i$ is a one dimensional function.  In addition, the angle between $\nabla u_i$ and $\nabla u_j$ is $\arccos\left(  \frac{|\partial_i H_j|}{\partial_i H_j} \right)$.  
\end{thm}

\begin{proof} Methods and ideas applied here are strongly motived by the ones given by Ambrosio and Cabr\'{e} in  \cite{ac} and  Alberti, Ambrosio and Cabr\'{e} in \cite{aac} in the case of a single equation and by Ghoussoub and the author in \cite{fg} for the case of systems.  We first note that $u$ being $H$-monotone means that  $u$ is a stable solution of (\ref{main}). Moreover, the function $v_i(x_1, x_2):=\lim_{x_3\to \infty}u_i(x_1, x_2, x_3)$ is also a bounded stable solution for (\ref{main}) in $\mathbb {R}^2$. Note also that since $u$ is an $H$-monotone solution, the system (\ref{main}) is then orientable. It follows from Theorem \ref{thm2} that each $v_i$ is one dimensional and consequently the energy of $v=(v_i)_{i=1}^m$ in a two-dimensional ball of radius $R$ is bounded by a multiple of $R$ which implies  
 \begin{equation}
\label{boundEt}
 \limsup_{t\to\infty} E(u^t) \le C R^{2},
 \end{equation}
 where $u^t(x'):=u(x',x_n+t)$ for $t\in\mathbb{R}$ and 
 $$E_R(u)=\int_{B_{R}} \sum_{i=1}^m \Phi(|\nabla u_i|^2)  - \tilde H(u)+ c_u ,
 $$ for $c_u:=\max \tilde H(u)$.  Applying similar arguments as in the proof of Theorem \ref{enopthm},  we shall show that  
\begin{equation}
\label{boundEphi}
 \sum_{i=1}^m \int_{B_{R}} \Phi(|\nabla u_i|^2) \le C R^{2},
 \end{equation}
 where the constant  $C$ is independent from $R$.  Note that shift function $u^t=(u_i^t)_{i=1}^m$ is also a bounded solution of (\ref{main}) with $|\nabla u_i^t| \in L^{\infty}(\mathbb{R}^n)$, i.e.,
\begin{eqnarray}
\label{maint}
-\div\left(\Phi'(|\nabla u^t_i|^2) \nabla u^t_i\right)=  H_i(u^t)   \ \ \text{in}\ \ \mathbb{R}^n,
  \end{eqnarray}
and also 
  \begin{eqnarray}
\label{monot}
\partial_t u^t_i>0 > \partial_t u^t_j \ \
\ \ \text{for all $i\in\bar I$ and $j\in \bar J$ and in} \  \mathbb{R}^n.
  \end{eqnarray}
Since  $u_i^t$ converges to $v_i$ in $C^1_{loc}(\mathbb {R}^n)$ for all $i=1,\cdots,m$, we have $$ \lim_{t\to\infty} E(u^t)= E(v).
$$
Now, we claim that the  following upper bound for the energy holds, for all  $ t \in\mathbb{R^+}$ 
\begin{equation}\label{energyt}
E_R(u) \le  E_R(u^t) +M \int_{\partial B_{R}} \left(  \sum_{i\in\bar I} (u_i^t-u_i)+ \sum_{j\in \bar J}(u_j-u_j^t)  \right)  dS ,
\end{equation}  
where $M:=\max_{i=1}^m \left\{|| \Phi'(|\nabla u_i|^2) |\nabla u_i|  ||_{L^\infty(\mathbb{R}^n)}\right\}$. Indeed, by differentiating the energy functional along the path $u^t$, one gets 
\begin{eqnarray}\label{der-energy}
\partial_t E_R(u^t)= \sum_{i=1}^m \int_{B_{R}} 2  \Phi'(|\nabla u^t_i|^2)   \nabla u_i^t\cdot \nabla (\partial _t u_i^t)- 2  H_i(u^t) \partial_t u_i^t,
    \end{eqnarray}
Now, multiply (\ref{maint}) with $\partial_t u^t$ and integrate by parts to get
  \begin{eqnarray}
\label{partialt}
&& \int_{B_{R}} \Phi'(|\nabla u^t_i|^2)  \nabla u_i^t\cdot \nabla (\partial _t u_i^t) 
\\&&\nonumber + \int_{\partial B_{R}}  \Phi'(|\nabla u^t_i|^2)  \partial_\nu u_i^t \partial_t u_i^t =\int_{B_{R}}   H_i(u^t) \partial_t u_i^t.
  \end{eqnarray}
for each $i=1,\cdots,m$.  From (\ref{partialt}) and  (\ref{der-energy}) we obtain 
  \begin{eqnarray}\label{der-energy-bd}
\partial_t E_R(u^t)= 2 \sum_{i}  \int_{\partial B_{R}}  \Phi'(|\nabla u^t_i|^2) \partial_\nu u_i^t \partial_t u_i^t.
   \end{eqnarray}
 Note that $-M\le  \Phi'(|\nabla u^t_i|^2) \partial_\nu u_i^t\le M$ for each $i$ and $\partial_t u_i^t>0>\partial_t u_j^t$ for $i\in \bar I$ and $j\in \bar J$ . Therefore, 
  \begin{eqnarray}\label{der-energy-bdM}
\partial_t E_R(u^t) \ge M \int_{\partial B_{R}}\left( \sum_{j\in \bar J} \partial_t u_j^t - \sum_{i\in \bar I}\partial_t u_i^t \right) d S.
   \end{eqnarray}
   On the other hand,
 \begin{eqnarray}\label{der-energy-last}
 E_R(u)&=& E_R(u^t)- \int_{0}^{t} \partial_t E_R(u^s) ds,\nonumber\\
 &\le&E_R(u^t) +M \int_{0}^{t}  \int_{\partial B_{R}} \left(  \sum_{i\in \bar I} \partial_s u_i^s-  \sum_{j\in \bar  J} \partial_s u_j^s \right) dS ds\nonumber\\
 &=&E_R(u^t) +M \int_{\partial B_{R}}  \left(  \sum_{i\in \bar I} (u_i^t-u_i)+ \sum_{j\in \bar J}(u_j-u_j^t)  \right) dS.
 \end{eqnarray}
To finish the proof of the theorem just note that $u_i<u_i^t$ and $u_j^t<u_j$ for all $i\in \bar I$, $j\in \bar J$ and $t\in \mathbb{R^+}$. Moreover, from (\ref{boundEt}) we have $\lim_{t\to\infty} E_R(u^t)\le CR^2$. Therefore, (\ref{der-energy-last})  yields  $ E_R(u) \le C |\partial B_R|\le C R^{2}$.  This proves  (\ref{boundEphi}).   

Now, set $\phi_i:=\partial_{x_n} u_i$ and $\psi_i:= \nabla u_i\cdot \eta$ where $\eta=(\eta',0)\in\mathbb R^{n-1}\times \{0\}$ and define $\sigma_i:=\frac{\psi_i}{\phi_i}$.  Lemma \ref{lmlinear} implies that $\sigma$ satisfies (\ref{insigma1}). Note that  $\phi_i^2\sigma^2_i=\psi_i^2 \le  |\nabla u_i|^2$.    From this and the fact that one of conditions (A) or (B) holds there exists a constant $M$ that is independent from $R$ such that
 \begin{eqnarray}\label{ddd}
\ \ \sum_{i=1}^{m}\int_{B_{2R}\setminus B_R} \phi_i^2\sigma_i^2
\mathcal A(\nabla u_i) {}_R\Gamma_i \cdot {}_R\Gamma_i &\le&  \sum_{i=1}^{m}\int_{B_{2R}} |\nabla u_i|^2 \mathcal A(\nabla u_i) {}_R\Gamma_i \cdot {}_R\Gamma_i 
\\&\le& M \sum_{i=1}^{m} \int_{B_{2R}} \Phi(|\nabla u_i|^2), 
 \end{eqnarray}
here we have used the fact that $||{}_R \Gamma_i(x)||_{L^\infty(\mathbb R^n)}, ||\nabla u_i||_{L^\infty(\mathbb R^n)}\le C$ for some  $C$ that is independent from $R$.  We now apply Proposition \ref{liouville} for $h_{ij}=\partial_j H_{i} (u) \phi_i \phi_j$ and the identity function $f$ to conclude that each $\sigma_i$ is constant. This implies that each $u_i$ is a one-dimensional function.  Since $u=(u_i)_{i=1}^m$ is a one-dimensional stable solution of (\ref{main}), Theorem   \ref{thm2} implies that the angle between $\nabla u_i$ and $\nabla u_j$ is $\arccos\left(  \frac{|\partial_i H_j|}{\partial_i H_j} \right)$. 

\end{proof}

\section{Liouville theorems  for symmetric systems}\label{secop}
  For bounded stable solutions of (\ref{main}) up to four dimensions we have the following Liouville theorem as long as each $H_i(u)$ is nonnegative.  

\begin{thm}\label{thmLio} Suppose that $u=(u_i)_{i=1}^m$ is a classical bounded stable solution of symmetric system (\ref{main}) where $H_i \ge 0$ for each $i$.   Let $\Phi$ satisfy one of conditions (A) or (B).   Assume also that $|\nabla u_i| \in L^\infty(\mathbb R^n) \cap W^{1,2}_{loc}(\mathbb R^n)$.  Then each $u_i$ must be constant provided  $n \le 4$. 
\end{thm}

\begin{proof} Multiply both sides of (\ref{main}) with ${}_R \zeta_i^2 [u_i - ||u_i||_{L^\infty(\mathbb R^n)}]$ and use  assumptions to get 
 \begin{equation} \label{divHp}
   -  {}_R \zeta_i^2 [u_i - ||u_i||_{L^\infty(\mathbb R^n)}] \div(\Phi'(|\nabla u_i|^2) \nabla u_i )   \le 0 \quad  \text{in} \ \  \mathbb{R}^n. 
  \end{equation} 
Applying integration by parts,  for each $i=1,\cdots,m$ we obtain 
 \begin{equation} \label{divHpint}
\int_{B_R}  \Phi'(|\nabla u_i|^2)  |\nabla u_i|^2   {}_R \zeta_i ^2  \le 2 \int_{B_R} \Phi'(|\nabla u_i|^2) |\nabla u_i | |\nabla {}_R \zeta_i| [||u_i||_{L^\infty(\mathbb R^n)}-u_i] {}_R \zeta_i
  \end{equation} 
From the Cauchy-Schwarz inequality, for any $R>1$,  we get  
\begin{equation} \label{phi'^2}
\sum_{i=1}^m \int_{B_R}  \Phi'(|\nabla u_i|^2)  |\nabla u_i|^2   \le C R^{n-2}. 
  \end{equation} 
Since $u$ is a stable solution of (\ref{main}) there exists a sequence of functions $\phi=(\phi_i)_{i=1}^m$ that each $\phi_i$ does not change sign. Similar to the proof of Theorem \ref{thmLio},  
 set $\psi_i:= \nabla u_i\cdot \eta$ where $\eta=(\eta',0)\in\mathbb R^{n-1}\times \{0\}$ and define $\sigma_i:=\frac{\psi_i}{\phi_i}$.  Lemma \ref{lmlinear} implies that $\sigma=(\sigma_i)_{i=1}^m$ satisfies (\ref{insigma1}). Note that  $\phi_i^2\sigma^2_i=\psi_i^2 \le  |\nabla u_i|^2$.    From this and the fact that one of conditions (A) or (B) holds there exists a constant $M$ that is independent from $R$ such that
 \begin{eqnarray}\label{dddphi2}
\nonumber \sum_{i=1}^{m}\int_{B_R} \phi_i^2\sigma_i^2
\mathcal A(\nabla u_i) {}_R\Gamma_i \cdot {}_R\Gamma_i &\le&  \sum_{i=1}^{m}\int_{B_{2R}} |\nabla u_i|^2 \mathcal A(\nabla u_i) {}_R\Gamma_i \cdot {}_R\Gamma_i 
\\&\le& M \sum_{i=1}^{m} \int_{B_{2R}} \Phi(|\nabla u_i|^2), 
 \end{eqnarray}
here we have used the fact that $||{}_R \Gamma_i(x)||_{L^\infty(\mathbb R^n)}, ||\nabla u_i||_{L^\infty(\mathbb R^n)}\le C$ for some  $C$ that is independent from $R$.   Note that due to the general  assumption $2s\Phi''(s)+\Phi(s)>0$ when $s>0$ and $\Phi(0)=0$ we have $0\le \Phi(s)\le 2 \Phi'(s) s$ for any $s>0$.  This implies that 
\begin{equation}\label{phit}
0\le \int_{B_{R}} \Phi(  |\nabla u_i|^2) \le 2\int_{B_{R}} \Phi'(  |\nabla u_i|^2) |\nabla u_i|^2 \ \ \text{in} \ \ \mathbb R^n.
  \end{equation}
From (\ref{phit}),  (\ref{dddphi2}) and (\ref{phi'^2}) we get 
\begin{equation}\label{dddphi3}
 \sum_{i=1}^{m}\int_{ B_R} \phi_i^2\sigma_i^2
\mathcal A(\nabla u_i) {}_R\Gamma_i \cdot {}_R\Gamma_i \le C R^{n-2}. 
  \end{equation}
We now apply Proposition \ref{liouville} for $h_{ij}=\partial_j H_{i} (u) \phi_i \phi_j$ and the identity function $f$ to conclude that each $\sigma_i$ is constant when $n\le 4$. This implies that each $u_i$ is a one-dimensional solution of (\ref{main}). Finally (\ref{dddphi3}) implies that each $u_i$ must be constant. 
\end{proof}

In the absence of stability condition, there are various Liouvile theorems for solutions of (\ref{main}), at least for the case of $m=1$, in \cite{ps,s} and references therein.   For the rest of this section,  we mainly focus on the $p$-Laplacian operator that is when $\Phi(s)=\frac{2}{p} s^{\frac{p}{2}}$ and radial solutions of (\ref{main}).  For this operator, we provide an optimal Liouville theorem for radial stable solutions.  The critical dimension is $n = \frac{4p}{p-1}+p$ that is much higher than $n=4$ given in Theorem \ref{thmLio} for not necessarily radial solutions. This implies that Theorem \ref{thmLio} does not seem to be optimal.   

Applying the definition of the $p$-Laplacian operator for radial functions in dimension $n$,  (\ref{main}) reads 
  \begin{equation} \label{radialsys}
   -|u_i'|^{p-2}\left( (p-1)u_i'' +\frac{n-1}{r} u_i'\right)  =   H_i(u)  \quad  \text{for} \ \  r \in \mathbb R^+. 
  \end{equation}   
Suppose that $u=(u_i)_{i=1}^m$ is a radial stable solution of (\ref{radialsys}) then in the light of (\ref{weakstab}) and (\ref{L}) we have 
  \begin{eqnarray*} \int_{\mathbb R^n} |\nabla u_i|^{p-2} (\nabla \phi_i,\nabla \zeta_i) &+& (p-2) \int_{\mathbb R^n} |\nabla u_i|^{p-4} (\nabla u_i,\nabla \phi_i)(\nabla u_i,\nabla \zeta_i) 
\\&=& \sum_{j=1}^m \int_{\mathbb R^n} \partial_j H_i(u)  \phi_j \zeta_i. 
\end{eqnarray*}
We now provide the stability inequality for solutions of (\ref{main}) with the $p$-Laplacian operator.  This is a particular case of Lemma \ref{stabineqphi}. 
\begin{lemma}\label{stabineq}  
  Let $u$ denote a stable solution of (\ref{main}).  Then 
\begin{equation} \label{stability}
\sum_{i,j=1}^{m} \int  \sqrt{\partial_j H_i(u) \partial_i H_j(u)} \zeta_i \zeta_j \le (p-1) \sum_{i=1}^{m} \int  |\nabla u_i|^{p-2} |\nabla \zeta_i|^2, 
\end{equation} 
for any $\zeta=(\zeta_i)_i^m$ where $ \zeta_i \in L^{\infty} (\mathbb R^n) \cap W^{1,2}(\mathbb R^n)$ with compact support and $1\le i\le m$. 
\end{lemma}  
For radial solutions stability inequality is of the following form. 
\begin{lemma} \label{stabineqgen}  
  Suppose that $u$ is a radial stable solution of  (\ref{main}).  Then 
\begin{eqnarray} \label{stabindep}
 &&(n-1) \sum_{i=1}^{m} \int_{\mathbb R^n}  \frac{u'^p_i(|x|)}{|x|^2} \phi^2(x) dx \le (p-1) \sum_{i=1}^{m} \int_{\mathbb R^n}    u'^p_i(|x|)   |\nabla \phi(x)|^2 dx  \\ &
& \label{tail}+    \sum_{i,j=1}^{m} \int_{\mathbb R^n}  \left(  \partial_j H_i(u) -\sqrt{\partial_j H_i(u) \partial_i H_j(u) } \right)  u'_i (|x|) u'_j (|x|) \phi^2(x) dx, 
\end{eqnarray} 
   for all $ \phi \in L^{\infty} (\mathbb R^n) \cap W^{1,2}(\mathbb R^n)$ with compact support. 

\end{lemma}

\begin{proof} Suppose that $u=(u_i)_{i=1}^m$ is a radial solutions of (\ref{main}) that is 
\begin{equation}\label{pderadial}
 -\frac{n-1}{r} |u'_i|^{p-2} u'_i - (p-1)  |u'_i|^{p-2} u_i''   =  H_i(u) ,  
  \end{equation} 
for $0<r<1$  and $i=1,\cdots,m.$   Multiplying the $i^{th}$ equation of (\ref{pderadial}) with $(u'_i \phi^2 r^{n-1})'$   for  $ \phi \in L^{\infty} (\mathbb R^n) \cap W^{1,2}(\mathbb R^n)$ with compact support  and performing integration by parts we obtain  
\begin{eqnarray}\label{www}
\int_{\mathbb R^+}   \sum_{j=1}^m  \partial_j H_i(u) {u'}_j u'_i \phi^2 r^{n-1} &=& -\int_{\mathbb R^+}  \left(\frac{n-1}{r}  |u'_i|^{p-2} u_i' \right)'  (u'_i \phi^2 r^{n-1}) 
\\&& \nonumber + (p-1) \int_{\mathbb R^+}   |u'_i|^{p-2} u_i'' (u'_i \phi^2 r^{n-1}) ', 
  \end{eqnarray}
  for all $0<r<1$ and  $ i=1,\cdots,m$.  In addition, straightforward calculations show that 
 \begin{eqnarray}\label{ff}
  (u'_i \phi^2 r^{n-1})'  &=&   (u'_i \phi^2 )' r^{n-1}+ (n-1) r^{n-2}   u'_i \phi^2, \\
  \label{fff}\left(\frac{n-1}{r}  |u'_i|^{p-2} u_i' \right)' &=& - \frac{n-1}{r^2}  |u'_i|^{p-2} u_i' + (p-1)\frac{n-1}{r}  |u'_i|^{p-2} u_i'' . 
   \end{eqnarray}
Substituting (\ref{ff}) and (\ref{fff}) in  (\ref{www}) we get 
  \begin{eqnarray}\label{www1}
\int_{\mathbb R^+}    \sum_{j=1}^m  \partial_j H_i(u) {u'}_j u'_i \phi^2 r^{n-1} &=& \int_{\mathbb R^+}   \frac{n-1}{r^2}  |u'_i|^{p} \phi^2 r^{n-1} 
\\&& \nonumber + (p-1) \int_{\mathbb R^+}   |u'_i|^{p-2} u_i'' (u'_i \phi^2 )'  r^{n-1} . 
  \end{eqnarray}
 Taking  sum  on the index $i$, we get 
  \begin{eqnarray}\label{wwww}
\int_{\mathbb R^n}   \sum_{i,j=1}^m  \partial_j H_i(u) {u'}_j u'_i \phi^2  &=& (n-1) \sum_{i=1}^m \int_{\mathbb R^n}  \frac{|u'_i|^{p}}{|x|^2}   \phi^2 
\\&& \nonumber + (p-1) \sum_{i=1}^m \int_{\mathbb R^n}   |u'_i|^{p-2}\nabla u_i'\cdot \nabla(u'_i \phi^2 )  .
  \end{eqnarray}
  
  We now apply the stability inequality (\ref{stability}) where $\phi_i$  is replaced by $ u'_i\phi$  for a test function   $\phi$.  Therefore, 
  \begin{equation} \label{stabsub}
  \sum_{i,j=1}^{m} \int_{\mathbb R^n} \sqrt{\partial_j H_i(u) \partial_i H_j(u)} u'_i u'_j\phi^2   \le (p-1) \sum_{i=1}^{m}  \int_{\mathbb R^n}  |u'_i|^{p-2}  |\nabla ( u'_i\phi)|^2 . 
\end{equation} 
  Expanding the integrand of the right-hand side we get 
  \begin{eqnarray*}\label{}
  (p-1)  |u'_i|^{p-2}  ( |\nabla ( u'_i\phi)|^2 &=&  (p-1)  |u'_i|^{p-2} \left( |u'_i|^2|\nabla \phi|^2 +  |\nabla u'_i|^2\phi^2+ \frac{1}{2} \nabla \phi^2\cdot\nabla {u'_i}^2 \right)
  \\&=& (p-1)  |u'_i|^{p-2} \left(|u'_i|^2 |\nabla\phi|^2+  \nabla (\phi^2 u_i') \cdot \nabla u_i'\right). 
  \end{eqnarray*}
  From this, (\ref{www}) and (\ref{stabsub}) we get the desired result. 
  
\end{proof}  

Now, we are ready to classify  radial stable solutions of (\ref{radialsys}). 

\begin{thm} \label{radialbound}
 Suppose that  $p,m\ge 1$ and $u$ is a radial stable solution of symmetric system (\ref{main})  where $H_i\in C^1(\mathbb R^m)$ whenever $\partial_j H_i(u)> 0$ for all $i,j=1,\cdots,m$.   Then, there exist positive constants $r_0$ and $C_{n,m,p}$ such that 
 \begin{equation}\label{lower}
 \sum_{i=1}^{m} | u_i(r)| \ge C_{n,m,p}   \left\{
                      \begin{array}{ll}
                        r^{\frac{1}{p} \left(p+2-n+2\sqrt{\frac{n-1}{p-1}}\right)}, & \hbox{if $n \neq \frac{4p}{p-1}+p$,} \\
                      \log r, &  \hbox{if $n =  \frac{4p}{p-1}+p$,} 
                           \end{array}
                    \right.
                    \end{equation}
                    where $r \ge r_0$ and $C_{n,m,p}$ is independent from $r$. In addition, assuming that each  $u_i$ is bounded  for $1\le i \le m$, then $n>\frac{4p}{p-1}+p$ and  there is a constant $C_{n,m,p}$ such that 
  \begin{equation}\label{limit}
\sum_{i=1}^{m} | u_i(r) - u_i^\infty| \ge C_{n,m,p}  r^{\frac{1}{p} \left(p+2-n+2\sqrt{\frac{n-1}{p-1}}\right)}, 
                    \end{equation}
                    where $r\ge 1$ and $u_i^\infty:=\lim_{r\to\infty} u_i{(r)}$ for each $i$.  
 \end{thm}

\begin{proof} Let $u=(u_i)_{i=1}^m$ be a radial stable solution of  symmetric system (\ref{radialsys}).  From Lemma \ref{stabineqgen}, %Therefore,  $$ \partial_j H_i(u) =\sqrt{\partial_j H_i(u) \partial_i H_j(u) }$$
  the stability inequality becomes 
 \begin{eqnarray} \label{stabindep11}
 &&(n-1) \sum_{i=1}^{m} \int_{\mathbb R^n}  \frac{u'^p_i(|x|)}{|x|^2} \phi^2(x) dx \le (p-1) \sum_{i=1}^{m} \int_{\mathbb R^n}    u'^p_i(|x|)   |\nabla \phi(x)|^2 dx , 
\end{eqnarray} 
   for all $ \phi \in L^{\infty} (\mathbb R^n) \cap W^{1,2}(\mathbb R^n)$ with compact support.  Note that the nonlinearity  $H=(H_i)_{i=1}^m$ does not appear in (\ref{stabindep11}).   The methods and idea that we apply in this proof are strongly motivated by the ones used in \cite{cc1,cc2,CS,svil,ces} for the case of a scalar equation, that is when $m=1$, and in \cite{mf,cf} for the case of system of equations that is when $m\ge 2$.    Test this inequality on the following test function $\phi\in W^{1,2}(\mathbb{R}^+)\cap L^{\infty} (\mathbb{R}^+)$ 
$$\phi(t):=   \left\{
                      \begin{array}{ll}
                        1, & \hbox{if $0\le t \le 1$;} \\
                       t^{-\sqrt{\frac{n-1}{p-1}}}, & \hbox{if $1\le t \le r$;} \\
                        \frac{    r^{-\sqrt{\frac{n-1}{p-1}}}   }{  \int_{r}^{R} \frac{dz}{z^{n-1} \sum_{i=1}^{m} {|u'_i|}^p(z)  } }    \int_{t}^{R} \frac{dz}{z^{n-1}  \sum_{i=1}^{m}  {|u'_i|}^p(z)  }  , & \hbox{if $r\le t \le R$;} \\
                          0, & \hbox{if $R\le t $,} 
                           \end{array}
                    \right.$$
                    for any $1\le r\le R$.   
                Straightforward calculations show that for the given test function $\phi$, the left-hand side of the stability inequality (\ref{stabindep11}) has the following lower bound,
\begin{equation}\label{LHS}
(n-1)  \int_{0}^{1} \sum_{i=1}^{m} {|u'_i|}^p(t)   t^{n-3} dt + (n-1)  \int_{1}^{r} \sum_{i=1}^{m} {|u'_i|}^p(t)   t^{   -2 \sqrt{\frac{n-1}{p-1}} +n-3} dt.
\end{equation}
Similarly we can simplify  the right-hand side of the stability inequality using the fact that 
$$\phi'(t)=   \left\{
                      \begin{array}{ll}
                        0, & \hbox{if $0\le t <1$;} \\
                      -\sqrt{\frac{n-1}{p-1}} t^{-\sqrt{\frac{n-1}{p-1}}-1}, & \hbox{if $1< t < r$;} \\
                     - \frac{    r^{-\sqrt{\frac{n-1}{p-1}}}   }{  \int_{r}^{R} \frac{dz}{z^{n-1} \sum_{i=1}^{m} {|u'_i|}^p(z)  } }     \frac{1}{t^{n-1}  \sum_{i=1}^{m}  {|u'_i|}^p(t)  }  , & \hbox{if $r\le t \le R$;} \\
                          0, & \hbox{if $R\le t $.} \\
                           \end{array}
                    \right.$$           
Substituting this in (\ref{stabindep11}), the right-hand side of the inequality would be  equivalent to
\begin{equation}\label{RHS}
(n-1) \int_{1}^{r}   t^{  -2\sqrt{\frac{n-1}{p-1}}+ n-3}  \sum_{i=1}^{m} {|u'_i|}^p(t)  dt +  \frac{    r^{-2\sqrt{\frac{n-1}{p-1}}   } }{  \int_{r}^{R} \frac{dz}{z^{n-1}\sum_{i=1}^{m} {|u'_i|}^p(z)   }  } .
\end{equation}
Collecting  (\ref{LHS}) and (\ref{RHS}), in the light of (\ref{stabindep11}), we get  
\begin{equation}\label{LR}
\int_{r}^{R} \frac{ds}{s^{n-1}\sum_{i=1}^{m} {|u'_i|}^p(s)   }  \le C_{n,m,p}  r^{-2\sqrt{\frac{n-1}{p-1}}}   \ \ \ \ \forall 1\le r\le R,
\end{equation}
where  the constant $C_{n,m,p}$ is independent from $r,R$ and it is  given as 
$$C_{n,m,p}:=\frac{p-1}{(n-1)  \int_{0}^{1} \sum_{i=1}^{m} {|u'_i|}^p(t)   t^{n-3} dt}.$$
   Applying the H\"{o}lder's inequality we obtain 
\begin{eqnarray}\label{l=r}
\ \ \ \ \ \  \int_r^R \frac{ds}{s^{\frac{n-1}{p+1}}} &= &\int_r^R \frac{\left(  \sum_{i=1}^{m} {|u'_i|}^p(s)   \right)^{\frac{1}{p+1}} }{s^{\frac{n-1}{p+1}} \left(  \sum_{i=1}^{m} {|u'_i|}^p(s)   \right)^{\frac{1}{p+1}}  } ds  
\\ & \le& \nonumber
 \left( \int_r^R  \frac{ds}{  s^{n-1}   \sum_{i=1}^{m} {|u'_i|}^p(s)   }         \right)^{\frac{1}{p+1}}   \left( \int_r^R   \left(\sum_{i=1}^{m} {|u'_i|}^p(s)     \right)^{\frac{1}{p}}    ds  \right)^{\frac{p}{p+1}} .   
\end{eqnarray}
From (\ref{LR})  we get 
\begin{equation}\label{l=r2}
\int_r^R \frac{ds}{s^{\frac{n-1}{p+1}}} \le C^{\frac{1}{p+1}}_{n,m,p}
\ r^{-\frac{p}{p+1}\sqrt{\frac{n-1}{p-1}}}    \left(  \sum_{i=1}^m \int_r^R    |u'_i(s) | ds \right)^{\frac{p}{p+1}} .
\end{equation}
Performing straightforward  computation for the integral in the left-hand side of (\ref{l=r2}) and taking $R=2r$, for any $n\ge 2$,  one can  get 
\begin{equation}\label{u2r}
 \sum_{i=1}^m | u_i(2r)-u_i(r)| \ge C_{n,m,p} r^{ \frac{1}{p} \left(  p+ 2-n +2 \sqrt{\frac{n-1}{p-1}}\right)}.
\end{equation}
Note that  each $u_i$ is bounded. Therefore,  (\ref{u2r}) implies 
 \begin{eqnarray*}\label{}
\sum_{i=1}^{m} |u_i(r)- u_i^\infty| &=& \sum_{i=1}^m \sum_{k=1}^\infty  | u_i(2^k r)-u_i(2^{k-1}r)| 
\\&\ge& C\sum_{k=1}^\infty  (2^{k-1}r)^{ \frac{1}{p} \left(  p+ 2-n +2 \sqrt{\frac{n-1}{p-1}}\right)}.
\end{eqnarray*}
 This proves the second part of the theorem that is (\ref{limit}) and $n> \frac{4p}{p-1}+p$.  To prove the first part of the theorem that is (\ref{lower}),  without loss of generality,  we assume that $2\le n\le  \frac{4p}{p-1}+p$.    Define $r=2^{k-1}r_1$ where $1\le r_1<2$. Therefore, 
\begin{eqnarray*}\label{}
\sum_{i=1}^{m} |u_i(r)| &=& \sum_{i=1}^m  | u_i(r)-u_i(r_1)| - \sum_{i=1}^m |u_i(r_1)| 
\\&=&\sum_{i=1}^m \sum_{j=1}^{k-1}  | u_i(2^j r_1)-u_i(2^{j-1}r_1)| - \sum_{i=1}^m |u_i(r_1)| \\&\ge&  C_{n,m} \sum_{i=1}^m \sum_{j=1}^{k-1} (2^{j-1}r_1)^{ ^{ \frac{1}{p} \left(  p+ 2-n +2 \sqrt{\frac{n-1}{p-1}}\right)}  } -\sum_{i=1}^m |u_i(r_1)|.
\end{eqnarray*}
This  proves the pointwise bound  (\ref{lower}) when $1< n< \frac{4p}{p-1}+p$. Finally, when we have  dimension $n= \frac{4p}{p-1}+p$,  from the above inequality,  we can prove 
\begin{eqnarray*}\label{}
\sum_{i=1}^{m} |u_i(r)| \ge C_{n,m}  (k-1)  - \sum_{i=1}^m |u_i(r_1)|.
\end{eqnarray*}
The fact that $k-1=\frac{\log r - \log r_1}{\log 2}$ completes the proof. 

\end{proof}

\end{document}